\documentclass[reqno, intlimits, sumlimits]{amsart} 

\usepackage{amsmath, amstext, amssymb, amsthm, bbm, lmodern, enumerate,graphicx} 

\usepackage{hyperref}
\hypersetup{colorlinks=true,linkcolor=black,citecolor=black,filecolor=black,urlcolor=black}
\usepackage[left=3cm,right=3cm, top=4cm, bottom=3cm]{geometry}

    \renewcommand{\phi}{\varphi}
    \renewcommand{\epsilon}{\varepsilon}

	\newcommand				{\eins}			{\mathbbm{1}}   
	\newcommand				{\norm}[1]		{\left\lVert#1\right\rVert}
	\newcommand				{\abs}[1]		{\left\lvert#1\right\rvert}
    
	\DeclareMathOperator	{\IE}			{\mathbb{E}} 
	
	\DeclareMathOperator	{\IP}			{\mathbb{P}}
	
	\DeclareMathOperator	{\IR}			{\mathbb{R}}

\theoremstyle{plain}
\newtheorem{thm}			{Theorem}
\newtheorem{lem}	[thm]	{Lemma}

\newtheorem{prop}	[thm]	{Proposition}
                     
\theoremstyle{definition}

\newtheorem{rem}	[thm]	{Remark}

\begin{document}
 \title{Fluctuations of the magnetization in the Block Potts Model}
\author{Jonas Jalowy}
 \address{Jonas Jalowy, Institute for Mathematical Stochastics, University of M\"unster, Germany}
  \email{jjalowy@wwu.de}
 \author{Matthias L\"owe}
 \address{Matthias L\"owe, Institute for Mathematical Stochastics, University of M\"unster, Germany.}
 \email{maloewe@math.uni-muenster.de}
 \author{Holger Sambale}
 \address{Holger Sambale, Faculty of Mathematics, Bielefeld University, Germany}
  \email{hsambale@math.uni-bielefeld.de}

 \date{\today}
 \thanks{Research of the first and second author was supported by the German Research Foundation (DFG) through the Excellence Strategy EXC 2044-390685587, Mathematics M\"unster: \emph{Dynamics-Geometry-Structure}.
 Research of the third author was funded by the Deutsche Forschungsgemeinschaft via the CRC 1283 \emph{Taming uncertainty and profiting from randomness and low regularity in analysis, stochastics and their applications}.}
 
 \begin{abstract}
In this note we study the block spin mean-field Potts model, in which the spins are divided into $s$ blocks and can take $q\ge 2$ different values (colors). Each block is allowed to contain a different proportion of vertices and behaves itself like a mean-field Ising/Potts model which also interacts with other blocks according to different temperatures. Of particular interest is the behavior of the magnetization, which counts the number of colors appearing in the distinct blocks. We prove central limit theorems for the magnetization in the generalized high temperature regime and provide a moderate deviation principle for its fluctuations on lower scalings. More precisely, the magnetization concentrates around the uniform vector of all colors with an explicit, but singular, Gaussian distribution. In order to remove the singular component, we will also consider a rotated magnetization, which enables us to compare our results to various related models.
 \end{abstract}
 
\maketitle

\section{Introduction}
In the early 2000s a statistical mechanics approach to social choice in several groups of interacting
agents lead to the development of so called block spin models \cite{Contucci_et_al08}. These block spin models were investigated in a number of papers, from both a static and
a dynamic point of view, see  \cite{gallo_contucci_CW}, \cite{gallo_barra_contucci}, \cite{fedele_contucci}, \cite{collet_CW}, \cite{LS18}, \cite{LSV20}, \cite{KLSS20}, \cite{KT20a}, \cite{KT20b}. Statistical questions in block spins models were studied in \cite{BRS19} and \cite{LS20}. 
All the above papers, however, discuss the
situation where each spin only takes two values, hence generalized mean-field Ising models or Curie-Weiss models. 

In \cite{Prepreprint} and \cite{Liu} for the first time block spin models 
with three or more possible spin values were considered. They are natural extensions of mean-field Potts models as considered as a generalization of the Curie-Weiss model in \cite{KeSchon89} and \cite{EW90}. 
In \cite{Liu} the author discusses the situation of a block spin 
Potts model with two blocks or groups and proves limit theorems for
this situation. In \cite{Prepreprint} on the other hand, a setting with several groups is analyzed, however, these have to have an 
equal size. In this case the authors are able to show a large
deviation principle, a phase transition, they are able to identify 
the limit points and prove logarithmic Sobolev inequalities. 
In particular, the authors in \cite{Prepreprint} identify a parameter
regime in the model that corresponds to the high temperature regime
in usual mean field spin models. In this regime the matrix-valued 
order parameter of the model has a unique limit point. 

The main goal of the present note is to prove a Central Limit Theorem (CLT, for short) for this order parameter (the block magnetization)
in the high temperature regime. Notice that Theorem 2 in \cite{Liu} studies a 
similar question. However, there only 
two blocks are allowed (which makes the 
situation technically considerably easier) and 
little is known about the parameter regime 
where the CLT holds. In the present paper we will study an arbitrary (finite) number of blocks of arbitrary block sizes with more general interactions possible. 

To be more specific, to present a simplified model and to fix some notation, partition $\{1, \ldots, N\}, N \in \mathbb{N}$ into $s$ blocks $S_k$. Let $\beta>0$ be the intra-block-interactivity within one of the blocks (for $s=1$, the parameter $\beta$ corresponds to the inverse temperature) and $\alpha\in (0, \beta)$ be the inter-block-interactivity between different blocks. Then, for any $q \in \mathbb{N}\setminus\{1\}$ (which can be thought of as the number of colors) and $\omega \in \{1, \ldots, q\}^N$, we consider the Hamiltonian
\begin{align}\label{eq:Hamiltonian}
H_N(\omega) := -\frac{\beta}{2N} \sum_{i \sim j} \mathbf{1}_{\omega_i = \omega_j} - \frac{\alpha}{2N} \sum_{i \nsim j} \mathbf{1}_{\omega_i = \omega_j},
\end{align}
where $i \sim j$ means that $i$ and $j$ belong to the same block $S_k$ and $i \nsim j$ means that they belong to different blocks. Here, for $i=j$ we have $i \sim i$, and for $i \ne j$ we count \emph{both} pairs $(i,j)$ and $(j,i)$. This gives rise to the Gibbs measure
\[
\mu_N(\omega) := \frac{\exp(-H_N(\omega))}{Z_{N}},
\]
where $Z_{N}$ is the partition function
\[
Z_N := \sum_{\omega'} \exp(-H_N(\omega')).
\]

This definition agrees with the one given in \cite{Prepreprint}, where the case of $q \ge 3$ was considered. Here we also include $q=2$, in which case we obtain a version of the Ising block model, cf.\ e.\,g.\ \cite{LS18} or \cite{KLSS20}. Note that our $\beta$ and $\alpha$ correspond to $2\beta$ and $2\alpha$ in \cite{LS18}, and similar remarks hold for \cite{KLSS20}. This is a consequence of using indicator functions rather than products $\omega_i\omega_j$ of spins $\omega_i, \omega_j \in \{\pm1\}$.

In the sequel it will turn out that our quantities of interest can be most handily 
described in term of the Kronecker product of matrices, denoted by $\otimes$ in the rest
of the paper. We define the magnetization of the block spin Potts model as the vector $m$ whose components are given by
\begin{align}\label{eq:MagnetizationDef}
m_{k,c}:=(e_k\otimes e_c)^Tm:=\frac 1 {\abs{S_k}}\sum_{i\in S_k}\eins_c(\omega_i),
\end{align}
where $k \in \{1, \ldots, s\}$, $c \in \{1, \ldots, q\}$ and $e_k \in \mathbb{R}^s$, $e_c \in \mathbb{R}^q$ are the respective standard unit vectors. Note that $m$ corresponds to the matrix $M_N$ in \cite{Prepreprint}, written as a vector, and under slight abuse of notation, we will still denote its components in a double index notation.
Denote by $I_d$ the identity matrix, by $\mathbf 1_{d\times e}$ the $d\times e$ matrix consisting only of ones and $\mathbf 1_d=\mathbf 1_{d\times 1}$. Define the diagonal matrix consisting of the block sizes $S= \operatorname{diag}( (\abs{S_k})_{k=1,\dots,s})$.

In the most general case we will treat in this note, the $s\times s$ block interaction matrix $A$ is assumed to be symmetric, positive definite and having only positive entries $A_{j,k}>0$. For instance, a sufficient condition for $A$ is if $A^{-1}$ is a so called Stieltjes matrix having negative off-diagonal entries. A more specific example of consideration will be the structured matrix $A_{\alpha,\beta}=\alpha\mathbf 1_{s\times s}+(\beta-\alpha) I_s$, which has $\beta$ on the diagonal and $\alpha$ elsewhere, corresponding to the Hamiltonian \eqref{eq:Hamiltonian}. We will denote by
\begin{align*}
\mathcal S= S\otimes I_q,\quad \mathcal A=A\otimes I_q
\end{align*}
the corresponding block structured matrices of the block sizes and interaction matrix, respectively.

Next let $\gamma_k=\lim_N\abs{S_k}/N$ and $\Gamma=\lim_{N\to\infty}\mathcal S /N=\operatorname{diag}(\gamma)\otimes I_q$.
Then the Hamiltonian can be rewritten as
\begin{align*}
H_N(\omega)=-\frac 1 {2N} m^T  \mathcal S\mathcal A \mathcal Sm
\end{align*}
Finally, let us recall the critical inverse temperature of the $q$-color block Potts model is given by $\zeta_q := 2 \frac{q-1}{q-2} \log(q-1) $. For $q=2$, this has to be read as $\zeta_2 = 2 = \lim_{q \downarrow 2} \zeta_q$ (see \cite{KeSchon89}, \cite{EW90}).
Our main result in the present paper 
is a CLT for the matrix of block magnetizations. If the block sizes
are asymptotically equal, this CLT reads as follows:  

\begin{thm}\label{thm:CLT}
Consider $A=A_{\alpha,\beta}$ in the \emph{high temperature case}, i.e. $$\lVert\sqrt\Gamma\mathcal A\sqrt\Gamma\rVert_2=\lim_{N\to\infty}\lVert\sqrt SA\sqrt S/N\rVert_2<\zeta_q,$$ of asymptotically equal block sizes, i.e. $\Gamma=I_{sq}/s$. 
Then, the normalized magnetization satisfies the following Central Limit Theorem
\begin{align*}
\sqrt{\mathcal S}(m-\tfrac 1 q\mathbf 1 _{sq})\Rightarrow W,
\end{align*} 
where $W\sim\mathcal N (0,\Sigma)$ is an $sq$-dimensional centered Gaussian random variable with \emph{singular} covariance matrix
\begin{align*}
 \Sigma=s\mathcal A^{-1}\left(\left(I_{sq}+\frac{\mathcal A} {s}\big( I_s\otimes(\tfrac 1{q^2}\mathbf 1_{q\times q}-\tfrac 1 q I_q)\big)\right)^{-1}-I_{sq}\right).
\end{align*}
\end{thm}

For non-homogeneous block sizes, we can prove a CLT under a more restrictive condition for the "inverse temperature" of a more general interaction matrix $A$.

\begin{thm}\label{thm:CLTgeneral}
If $\gamma_k>0$ for all $k=1,\dots,s$ and $\lVert\sqrt \Gamma \mathcal A\sqrt \Gamma\rVert_2<4\frac{q-1}q$, then the normalized magnetization satisfies the following Central Limit Theorem
\begin{align*}
\sqrt{\mathcal S}(m-\tfrac 1 q\mathbf 1 _{sq})\Rightarrow W,
\end{align*} 
where $W\sim\mathcal N (0, \Sigma  )$ is an $sq$-dimensional centered Gaussian random variable with \emph{singular} covariance matrix
\begin{align*}
\Sigma&=\big(\sqrt\Gamma\mathcal A\sqrt\Gamma\big)^{-1}\left(\Big(I_{sq}+\sqrt\Gamma\mathcal A\sqrt\Gamma\big( I_s\otimes(\tfrac 1{q^2}\mathbf 1_{q\times q}-\tfrac 1 q I_q)\big)\Big)^{-1}-I_{sq}\right)
\end{align*}
\end{thm}
 
\begin{rem}
The value $4\frac{q-1}q$ defines a threshold between different cases in \cite[Proof of Proposition 2.1]{EW90} as well, however it is not possible to generalize their monotonicity (in $\beta$) arguments to $s>1$ blocks. Instead, our proof relies on a fixed-point argument, as in \cite{KLSS20,LS18}. This argument does not apply above the given threshold, where additional local minima appear already in the classical Potts model, and hence we cannot reach the critical inverse temperature $\zeta_q$. Note that $4\frac{q-1}q\le\zeta_q$ with equality for $q=2$. Moreover, the same spectral norm also distinguishes between the high and low temperature case in \cite{KLSS20} and \cite{DebMuk}.

For $s=2$, Theorem \ref{thm:CLT} is an explicit version of \cite[Theorem 2]{Liu} and the limiting distribution coincides. However, it is not apparent from \cite{Liu} that the limiting distribution is degenerate, which we shall circumvent in the following.

Moreover, if we replace the condition $\lVert{\sqrt \Gamma \mathcal A\sqrt \Gamma}\rVert_2<4\frac{q-1}q$ by the same abstract condition as in \cite[Equation (16)]{Liu}, i.\,e.\ if $\phi$ has a unique minimizer $\xi^*$ at which the Hessian is strictly positive definite and if $\phi$ grows at least quadratically far away from $\xi^*$, then our CLT continues to hold (since this is the last part of the proof from \eqref{eq:Taylor} and below).
Here, $\phi$ is given in Lemma \ref{lem:phi} below.

\end{rem}

\begin{rem}\label{rem:Rotation}
The components of $m$ are coupled since $m_{k,q}=1-\sum_{c=1}^{q-1} m_{k,c}$. Therefore $\Sigma$ is singular and has exactly $s$ many zero eigenvalues. We shall circumvent this by projecting to the first $q-1$ coordinates of each block after rotating the hyperplane of the support of the limiting Gaussian distribution. We rotate by the matrix $I_s\otimes\tilde R\in SO(sq)$, for which each $s\times s$ diagonal-block is given by the matrix $\tilde R$ which rotates the normal vector of the support to the unit vector, i.e. $\tilde R \frac 1 {\sqrt q}\mathbf 1 _{q}=e_q$. If $\tilde P$ is the $(q-1)\times q$-projection matrix onto the first $(q-1)$ coordinates, define $\mathcal R=I_s\otimes\tilde P \tilde R$. According to the previous discussion, it is reasonable to consider the rotated (as well as projected and rescaled) magnetization given by 
\begin{align}\label{eq:RotMag}
    \hat m =\mathcal  R\sqrt{\mathcal S}(m-\tfrac 1 q\mathbf 1 _{sq})=(\operatorname{diag}(\sqrt{\abs{S_k}}_{k=1,\dots, s})\otimes \tilde P \tilde R) m.
\end{align}
Note that the rotated magnetization $\hat m$ is already centered because of $\mathcal  R \mathbf 1_{sq}=0$.
\end{rem}

\begin{thm}\label{thm:CLTrot}
Under the conditions of Theorem \ref{thm:CLT} or Theorem \ref{thm:CLTgeneral}, respectively, the rotated magnetization $\hat m$ satisfies the Central Limit Theorem
\begin{align*}
    \hat m\Rightarrow\mathcal N (0 , (q-A/s)^{-1}\otimes I_{q-1})
\end{align*}
in the case of asymptotically equal block sizes or, more generally
\begin{align*}
\hat m\Rightarrow\mathcal N \Big(0 , \big(q-\operatorname{diag(\gamma)}^{1/2}A \operatorname{diag(\gamma)}^{1/2}\big)^{-1}\otimes I_{q-1}\Big),
\end{align*}
where the covariance matrix is non-singular and strictly positive definite.
\end{thm}

This rotated CLT makes a comparison to the mean-field Ising model, or its block model version, more apparent. We would like to emphasize that in the case of the Curie-Weiss model, we take $q=2,s=1$, and $\hat m= \sqrt{N/2}(m_2-m_1)$ is the usual magnetization of the Ising model up to a factor of $1/\sqrt 2$. The limiting distribution is given by $\mathcal N (0 , (2-\beta)^{-1})$ which coincides with the classical results, see \cite[Theorem V.9.4]{EllisBook}. Note that once again, our $\beta$ corresponds to $2\beta$ in \cite{EllisBook}. Moreover, one may retrieve \cite[Theorem 1.2]{LS18} in the case of $q=s=2$, where our $\beta$ corresponds to $4\beta$ of \cite{LS18}.

The object of study in the LLN/LDP in \cite{KLSS20} differs from the CLT's above by the scaling of magnetization. The following moderate deviation principle (MDP) describes the fluctuation on scalings transitioning between the two.

\begin{thm}\label{thm:MDP}
Assume the conditions of Theorem \ref{thm:CLT} or Theorem \ref{thm:CLTgeneral}. For any $0<\theta<1/2$, the distribution of $\mathcal R\mathcal S^{\theta}(m-\frac 1 q \mathbf 1_{sq})$ under the Gibbs measure $\mu_{N}$ satisfies a moderate deviation principle of speed $N^{1-2\theta}$ with good rate function 
\begin{align*}
\Lambda(t)=\frac 1 2 t^T \left(\left(q\operatorname{diag(\gamma)}^{1-2\theta}-\operatorname{diag(\gamma)}^{1-\theta}A \operatorname{diag(\gamma)}^{1-\theta}\right)\otimes I_{q-1}\right)t.
\end{align*}
\end{thm}

As usual in large deviations theory, the large deviations rate function derived in 
\cite[\S 3]{Prepreprint}
is not obtained when setting $\theta=0$ in Theorem \ref{thm:MDP}, while for 
$\theta=1/2$ one obtains 
as rate function the exponent of the CLT in 
 Theorem \ref{thm:CLT} and Theorem \ref{thm:CLTgeneral}. 
 It is folklore knowledge in large and moderate deviation theory that such a smooth transition from the MDP regime to the regime of the CLT is usually true, even though there are counterexamples (see e.g. \cite{LM12} for a counterexample). 
 
 Let us yet again consider the special case $q=2$, $s=1$ corresponding to the classical Curie-Weiss model. Here, $\Lambda(t)=(2-\beta)t^2/2$, which coincides with \cite[Example 2.1]{MDP}, after noticing that our $\beta$ corresponds to $2\beta$ in \cite{MDP}.

The proofs of the results presented above are based on the Hubbard-Stratonovich transform, or in other words, convoluting the distribution of the rescaled magnetization under the Gibbs measure with a suitable Gaussian measure will lead to a representation where Laplace's Method can be applied. Let us first state how this convolution looks like. In particular, the function $\phi$ (cf. \eqref{eq:phi} below) in the exponent of the resulting density will be of major interest.

\begin{lem}\label{lem:phi}
For any $0\le\theta\le 1/2$ and all $v\in\IR ^{sq}$ we have
\begin{align}\label{eq:distribution}
 &\mu_N\circ \left(\mathcal S^\theta(m-v)\right)^{-1}\ast\mathcal N \left(0,N( \mathcal S^{1-\theta}\mathcal A \mathcal S^{1-\theta})^{-1}\right)(d^{sq}x)\nonumber\\
 &=c_N \exp\left[- N \phi\left(\left(\frac{\mathcal S}{ N}\right)^{1-\theta}\frac x { N^{\theta}}+\frac {\mathcal S}{N}v \right)\right]d^{sq}x,
\end{align}
where 
\begin{align*}
 \phi(\xi)=\tfrac 1{2} \xi^T\mathcal A \xi-\sum_{k=1}^s \frac{\abs{S_k}}{N} \log\left(\sum_{c=1}^q \exp\left[\xi^T\mathcal A(e_k\otimes e_c) \right]\right).
\end{align*}\end{lem}
Since $\phi$ depends on $N$ in the case of non-equal Block sizes $\abs{S_k}\not\equiv N/s$, we shall write $\phi_N$ whenever this dependency needs to be stressed. Note that this is the analogue of \cite[Lemma 3.2]{EW90} in the case of $\theta=1/2, s=1$ and that the sum of exponential functions corresponds to $2\cosh(\cdot)$ in the Ising model, cf.\ \cite[Proof of Lemma 3.1]{LS18}.

Finding the global minima of $\phi$ will be crucial, hence we shall investigate $\phi$ in the following. In Section \ref{sec:convex} we build a bridge to the convex dual of $\phi$ which is studied in \cite{Prepreprint}.
In order to do so, we would have liked to invoke classic results from the theory of classic duality as given e.g.\ by Eisele and Ellis 
\cite[Appendix C]{EE83}. However, our function of interest fails to meet the assumptions made in this theory. This is repaired by introducing subdifferentials and by
applying the theory of convex duality to submanifolds. Note that a similar result can be found in \cite[Theorem A.1]{CET}, using related methodology. Nevertheless, we chose to include our own arguments for completeness and to keep this paper self-contained. Moreover, we do not only provide a result for minimizer but for all critical points and explicitly state the bijection which the duality is based on.

In Section \ref{sec:Minimizer} we prove Lemma \ref{lem:phi}. Moreover, we study the properties of $\phi$ in detail and, in particular, find its minima. For asymptotically equal block sizes we can make use of results obtained in \cite{Prepreprint}. This is why in this case we obtain results up to the critical temperature, while in the case of general block sizes and general interaction matrix we need to study $\phi$ directly. Finally, in Section \ref{sec:proofs} we will prove the main theorems of this note.

\section{Some facts on convex duality}\label{sec:convex}
Recall that the \emph{convex conjugate} of a function $f \colon \mathbb{R}^d \to \mathbb{R} \cup \{\infty\}$ is defined by
\begin{align*}
f^*(\nu)=\sup_{\xi\in \mathbb{R}^d}\{\xi^T\nu-f(\xi)\} = \sup_{\xi\in \mathrm{dom}(f)}\{\xi^T\nu-f(\xi)\},
\end{align*}
where $\mathrm{dom}(f) := \{\xi\in\mathbb{R}^d \colon f(\xi) < \infty\}$ (in passing, note that we never allow for $f(\xi) = - \infty$). According to \cite[Theorem 12.2]{Ro70}, this transform is an involution $(f^*) ^*=f$ for closed convex functions.

For the ease of presentation, we first show a duality result in a simplified situation, i.\,e.\ for functions which behave particularly well with respect to convex conjugates. That is, assume that $B := \mathrm{int}(\mathrm{dom}(f))$ is a non-empty open convex set and that $f$ is strictly convex and essentially smooth, i.\,e.\ $f$ is differentiable on $B$ and $\lim|\nabla f(x_i)| = \infty$ whenever $x_i$ is a sequence in $B$ converging to a boundary point of $B$. Such a function is also called a convex function of Legendre type. In this case, by \cite[Theorem 26.5]{Ro70}, $f^*$ is a convex function of Legendre type as well, we have
\begin{align}\label{eq:nablainv*}
\nabla f \circ\nabla f^*=Id=\nabla f^*\circ\nabla f,
\end{align}
and by \cite[Theorem 23.5]{Ro70},
\begin{align}\label{eq:argsup*}
\operatorname{argsup}_{\nu}\{\nu^T\xi-f^*(\nu)\}=\nabla f (\xi).
\end{align}
Using these basic relations, we obtain the following lemma:

\begin{lem}\label{lem:convexduality*}
Let $f,g$ be convex functions of Legendre type. Then we have the convex duality
\begin{align}\label{eq:convexduality*}
    \sup_{\xi\in\mathrm{dom}(g)}\{f(\xi)-g(\xi)\}=\sup_{\nu\in\mathrm{dom}(f^*)}\{g^*(\nu)-f^*(\nu)\}.
\end{align}
Moreover, the bijection $\nabla g$, given by \eqref{eq:nablainv*}, is 1:1 between critical points and 1:1 between maximizers of \eqref{eq:convexduality*}, and for any critical point $\xi$ of $f-g$, we have
\[
f(\xi)-g(\xi) = g^*(\nabla g(\xi))-f^*(\nabla g(\xi)).
\]
\end{lem}

In particular, if the right supremum is attained at a unique maximizer $\nu^*$, then the left supremum is attained at the unique maximizer $\xi^*=\nabla g^*(\nu^* )$.

\begin{proof}
The duality \eqref{eq:convexduality*} can be found in Eisele and Ellis \cite[Appendix C]{EE83}. 

Suppose that $\xi$ is a critical point of $f-g$. In particular,
\begin{align}\label{eq:critpoint*}
\nabla (f (\xi)-g(\xi))=0,
\end{align}
i.\,e.\ $\nabla f$ and $\nabla g$ take the same value on critical points, and hence by \eqref{eq:nablainv*},
\begin{align*}
\nabla f (\nabla g^*(\nabla g(\xi)))=\nabla f(\xi)=\nabla g(\xi).
\end{align*}
Applying $\nabla f^*$ to both sides reveals that $\nabla g(\xi)$ is a critical point of $g^*-f^*$. 
Moreover, it follows from \eqref{eq:critpoint*} and the representation \eqref{eq:argsup*} that
\begin{align}\label{eq:Id1*}
   g^*(\nabla g(\xi))-f^*(\nabla g(\xi))= g^*(\nabla g(\xi))-f^*(\nabla f(\xi))+\xi^T(\nabla f(\xi)-\nabla g(\xi))=f(\xi)-g(\xi).
\end{align}
Conversely, using the above reasoning together with the involution property of the Legendre transform, any critical point $\nu$ of $g^*-f^*$ is also a critical point of $f-g$ and satisfies
\begin{equation}\label{eq:Id2*}
    g^*(\nu)-f^*(\nu) = f(\nabla g^*(\nu)) - g(\nabla g^*(\nu)).
\end{equation}

Suppose in addition that $\xi$ is a maximizer such that $f(\xi)-g(\xi)=\sup \{f-g\}$. Then, by convex duality \eqref{eq:convexduality*} and \eqref{eq:Id1*}, we obtain that $\nabla g(\xi)$ is a maximizer of $g^*-f^*$. Conversely let $\nu$ be a maximizer of $g^*-f^*$, then by \eqref{eq:convexduality*} and \eqref{eq:Id2*} also $\nabla g^*(\nu)$ is a maximizer of $f-g$.
\end{proof}

In the proofs of Theorem \ref{thm:CLT} and Theorem \ref{thm:CLTgeneral}, we need a result of the same type as Lemma \ref{lem:convexduality*} applied to the functions
\begin{equation}\label{fandg*}
f(\xi) := \sum_{k=1}^s\frac{|S_k|}{N}\log\sum_{c=1}^q\exp(\xi_{k,c}),\qquad g(\xi):= \frac 12\xi^T\mathcal A ^{-1}\xi
\end{equation}
such that $\phi(\xi)=g(\mathcal A \xi)-f(\mathcal A \xi)$. Here, and in the sequel, we used the same notation as for the magnetization \eqref{eq:MagnetizationDef}, i.e. we denote the components by $\xi_{k,c}:=(e_k\otimes e_c)^T\xi=\xi_{(k-1)q+c}$.
Let us calculate the convex conjugates $f^*$ and $g^*$. First, note that $g^*(\nu)=\frac 1 {2}\nu^T\mathcal A \nu$, see \cite[Example VI.5.1]{EllisBook}. Therefore, it remains to find
\begin{align}\label{eq:f=sup*}
f^*(\nu)=\sup_{\xi} \{\nu^T\xi  - f(\xi )\}.
\end{align}
Differentiation with respect to $\xi$ shows that the supremum is attained at a critical point satisfying
 \begin{align*}
     \nu_{k,c}-\frac{|S_k|}{N}\frac{\exp( \xi _{k,c})}{\sum_{\tilde c}\exp( \xi _{k,\tilde c}) }=0
 \end{align*}
for all $c=1,\dots,q$ and $k=1,\dots,s$. In particular, critical points can only exist if $\nu_{k,c} > 0$ and moreover, summing over $c=1,\dots,q$ shows that we must have $\sum_{c=1}^q\nu_{k,c}=|S_k|/N$ for all $k=1,\dots,s$. (If these conditions are not satisfied, it is not hard to see that $f^*(\nu) = \infty$.)

Applied to \eqref{eq:f=sup*}, this readily implies $f^*(\nu)=\nu^T\log \nu-H$, where the logarithm is considered componentwise and $\mathcal H = \mathcal H(|S_\bullet|/N) = \sum_{k=1}^s(|S_k|/N)\log(|S_k|/N)$ is the entropy of $(|S_1|/N, \ldots, |S_s|/N)$. Summing up, we have
\begin{equation}\label{convexconjugates*}
f^*(\nu) = \begin{cases} \sum_{c=1}^q \nu^T\log \nu-\mathcal H, & \nu \in C\\ \infty, &\text{else}\end{cases},\qquad g^*(\nu)=\frac 1 {2}\nu^T\mathcal A \nu,
\end{equation}
where
\begin{equation}\label{DasIstC}
C := \{\nu \in \IR^{sq} \colon \nu_{k,c} > 0\ \text{and}\ \sum_{c=1}^q \nu_{k,c} = |S_k|/N\ \textup{for all}\ k\}.
\end{equation}
In particular, note that in this situation, we cannot directly apply Lemma \ref{lem:convexduality*}, since $f$ is not a convex function of Legendre type. In fact, its convex conjugate $f^*$ is not even differentiable (as a function on $\mathbb{R}^{sq}$). Therefore, we need to modify the arguments leading to Lemma \ref{lem:convexduality*}.

To this end, let us recall the notion of subdifferentials. Let $h \colon \mathbb{R}^{sq} \to \mathbb{R} \cup \{ \infty\}$ be a convex function. Then, the subdifferential $\partial h(\nu)$ of $h$ in $\nu \in \mathbb{R}^{sq}$ is the the set of all vectors (subgradients) $\nabla h(\nu) \in \mathbb{R}^{sq}$ such that
\begin{equation}\label{subdiff*}
h(\eta) \ge h(\nu) - \langle \nabla h(\nu), \eta-\nu \rangle
\end{equation}
for all $\eta \in \mathbb{R}^{sq}$ or, equivalently, all $\eta \in \mathrm{dom}(h)$. Labeling the elements of $\partial h(\nu)$ by $\nabla h(\nu)$ is non-standard but convenient for our purposes. If $h$ is differentiable in $\nu$, then $\partial h(\nu)$ just consists of a single element $\nabla h(\nu)$, which is the usual Euclidean gradient of $h$ at $\nu$, cf.\ \cite[Theorem 25.1]{Ro70}. Assuming $h$ to be a closed function, \eqref{eq:argsup*} continues to hold for subgradients, and generalizing \eqref{eq:nablainv*}, we have
\begin{align}\label{nablainvgen*}
\begin{split}
\nabla h(\nu) \in \partial h(\nu) &\Leftrightarrow \nu \in \partial h^*(\nabla h(\nu)),\\
\nabla h^*(\nu) \in \partial h^*(\nu) &\Leftrightarrow \nu \in \partial h(\nabla h^*(\nu))
\end{split}
\end{align}
according to \cite[Theorem 23.5]{Ro70}. In particular, we can always choose suitable elements of the respective subdifferentials such that \eqref{eq:nablainv*} holds (pointwise). For illustration, consider $f$ and $g$ as in \eqref{fandg*} for $q=2$ and $s=1$. In this case,
\begin{align*}
\partial f(\xi) = \Big(\frac{e^{\xi_1}}{e^{\xi_1} + e^{\xi_1}}, \frac{e^{\xi_2}}{e^{\xi_1} + e^{\xi_1}}\Big),\qquad \partial g(\xi) &= (\xi_1/\beta, \xi_2/\beta)\\
\partial f^*(\nu) = \begin{cases} \{ (z, z + \log(\nu_2/\nu_1) \colon z \in \mathbb{R} \} & \nu \in (0,1)^2, \nu_1 + \nu_2 = 1\\ \emptyset &\text{else}\end{cases},\qquad \partial g^*(\nu) &= (\beta \nu_1, \beta \nu_2).
\end{align*}

In Lemma \ref{lem:convexduality*}, the supremum of $g^*-f^*$ is evaluated over $\mathrm{dom}(f^*)$, which in our case equals $C$ as in \eqref{DasIstC}. Therefore, we have to consider functions $h \colon C \to \mathbb{R}$. Set
\begin{equation}\label{DasIstCHut}
\hat{C} := \{\hat{\nu} \in \IR^{(q-1)s} \colon \hat{\nu}_{k,c} > 0\ \text{and}\ \sum_{c=1}^{q-1} \hat{\nu}_{k,c} < |S_k|/N\ \textup{for all}\ k\}.
\end{equation}
Obviously, there is a natural bijection between $C$ and $\hat{C}$: any $\nu \in C$ corresponds to a $\hat{\nu}$ whose $k$-th block is given by $(\nu_{k,1}, \ldots, \nu_{k,q-1})$ (i.\,e.\ the first $q-1$ coordinates of the $k$-th block of $\nu$). In other words, recalling the projection $\tilde P$ from Remark \ref{rem:Rotation}, we have $\hat{\nu}=(I_s\otimes \tilde P)\nu\in \hat{C}$.
Conversely, given $\hat{\nu}$, we get back $\nu$ by adding a $q$-th coordinate to each block by setting $\nu_{k,q} = |S_k|/N-\sum_{c=1}^{q-1}\hat{\nu}_{k,c}$. In particular, any function $h \colon C \to \mathbb{R}$ naturally gives rise to a function $\hat{h} \colon \hat{C} \to \mathbb{R}$ given by $\hat{h}(\hat{\nu}) := h(\nu)$, where $\nu$ is the element of $C$ corresponding to $\hat{\nu}$ as described above. In the same way, any function $\hat{h}$ on $\hat{C}$ naturally corresponds to a function $h$ on $C$. The subdifferentials of $h$ and $\hat{h}$ are related as follows:

\begin{lem}\label{lem:SD*}
Consider $C$ as in \eqref{DasIstC} and $\hat{C}$ as in \eqref{DasIstCHut}, let $h \colon \mathbb{R}^{sq} \to \mathbb{R}$ be a convex function with $C \subset \mathrm{dom}(h)$ and $\hat{h}$ as introduced above. For any vector $\xi \in \mathbb{R}^{sq}$, let $\widetilde{\xi} \in \mathbb{R}^{(q-1)s}$ be the vector whose $k$-th block is given by
\begin{equation}\label{tilde}
    (\xi_{k,c}-\xi_{k,q})_{c=1}^{q-1},
\end{equation}
i.\,e.\ the $q$-th component of each block is subtracted from all the other components.
\begin{enumerate}
    \item $\hat{h}$ is a convex function on $\hat{C}$, and for any $\nu \in C$ and any $\nabla h(\nu) \in \partial h(\nu)$, the vector $\widetilde{\nabla h(\nu)}$ lies in $\partial \hat{h}(\hat{\nu})$.
    \item Assume $\mathrm{dom}(h) = C$, i.\,e.\ $h(\nu) = \infty$ if $\nu \notin C$. If $\hat{\nu} \in \hat{C}$ and $\nabla \hat{h}(\hat{\nu}) \in \partial \hat{h}(\hat{\nu})$, then any vector $x \in \mathbb{R}^{sq}$ such that $\widetilde{x} = \nabla \hat{h}(\hat{\nu})$ lies in $\partial h(\nu)$.
\end{enumerate}
\end{lem}

\begin{proof}
This follows readily from the definition of the subdifferential as in \eqref{subdiff*}. Indeed, to see $(1)$, take any $\hat{\eta} \in \hat{C}$. Then, writing $\nu_{k,q} = |S_k|/N-\sum_{c=1}^{q-1}\hat{\nu}_{k,c}$, we have
\begin{align*}
    \hat{h}(\hat{\eta}) = h(\eta) &\ge h(\nu) - \langle \nabla h(\nu), \eta - \nu \rangle\\
    &= \hat{h}(\hat{\nu}) - \langle \widehat{\nabla h(\nu)}, \hat{\eta} - \hat{\nu} \rangle - \sum_{k=1}^s (\nabla h(\nu))_{k,q}\sum_{c=1}^{q-1} (\nu_{k,c}-\eta_{k,c})\\
    &=\hat h (\hat \nu)-\langle \widetilde{\nabla h (\nu)},\hat\eta-\hat\nu\rangle,
\end{align*}
which is what had to be proven. Reversing these arguments immediately leads to $(2)$, where the condition $\mathrm{dom}(h) = C$ guarantees that we may restrict ourselves to vectors $\eta \in C$.
\end{proof}

We are now ready to prove an analogue of Lemma \ref{lem:convexduality*} for the functions $f$ and $g$ from \eqref{fandg*}.

\begin{lem}\label{convexdualityfg*}
The statement of Lemma \ref{lem:convexduality*} continues to hold for $f$ and $g$ as in \eqref{fandg*}.
\end{lem}

Note that in the case, a critical point of $g^*-f^*$ has to be understood as a critical point of $\hat{g}^* - \hat{f}^*$, which is a differentiable function on $\hat{C}$. To switch between the differentials of $g^*$ and $\hat{g}^*$, we may then use Lemma \ref{lem:SD*} (1) and the uniqueness of the subdifferential for differentiable functions. In particular, once again it holds that if the right supremum is attained at a unique maximizer $\nu^*$, then the left supremum is attained at the unique maximizer $\xi^*=\nabla g^*(\nu^* )$.

\begin{proof}
We follow the proof of Lemma \ref{lem:convexduality*} and adapt it to the situation under consideration. Obviously, \eqref{eq:convexduality*} holds since for this relation, no differentiability assumptions are needed. For the rest of the proof, note that for $f$ and $g$ as in \eqref{fandg*}, the functions $f$, $g$, $g^*$, $\hat{f}^*$ and $\hat{g}^*$ are differentiable, which in particular yields uniqueness of the corresponding subdifferentials.

Suppose that $\xi$ is a critical point of $f-g$. In particular,
\begin{align}\label{eq:critpointunten*}
\nabla (f (\xi)-g(\xi))=0,
\end{align}
i.\,e.\ $\nabla f$ and $\nabla g$ take the same value on critical points, and hence by \eqref{eq:nablainv*},
\begin{align*}
\nabla f (\nabla g^*(\nabla g(\xi)))=\nabla f(\xi)=\nabla g(\xi).
\end{align*}
Now choose any $\nabla f^*(\nabla g(\xi)) \in \partial f^*(\nabla g(\xi))$ and apply it to both sides of this equality, which yields
\begin{align*}
\nabla f^*(\nabla f (\nabla g^*(\nabla g(\xi))))=\nabla f^*(\nabla g(\xi)),
\end{align*}
or, in terms of $\hat{f}^*$,
\begin{align}\label{relhat}
\nabla \hat{f}^*\big((I_s\otimes\tilde P)\nabla f (\nabla g^*(\nabla g(\xi)))\big)=\nabla \hat{f}^*(\widehat{\nabla g(\xi)}),
\end{align}
where for convenience of notation we used $I_s\otimes \tilde P$ instead of the $\widehat{\cdot}$ notation.

Now note that for any $\omega \in \mathbb{R}^{sq}$, $\nabla \hat{f}^*(\widehat{\nabla f(\omega)}) = \widetilde{\omega}$. To see this, note that by \eqref{nablainvgen*}, a suitable choice of subgradient yields $\nabla f^*(\nabla f(\omega)) = \omega$, which by Lemma \ref{lem:SD*} (part (1), for $h = f^*$, $\nu = \nabla f(\omega)$) yields the claim for this choice of subgradient, from where the claim also follows in general since $\hat{f}^*$ is differentiable (and thus has a unique subdifferential). In particular, the left-hand side of \eqref{relhat} reads $\widetilde{\nabla g^*(\nabla g(\xi))}$, which by Lemma \ref{lem:SD*} (1) equals $\nabla \hat{g}^*(\widehat{\nabla g(\xi)})$. Altogether, $\nabla \hat{f}^*(\widehat{\nabla g(\xi)}) = \nabla \hat{g}^*(\widehat{\nabla g(\xi)})$, i.\,e.\ $\widehat{\nabla g(\xi)}$ is a critical point of $\hat{g}^* - \hat{f}^*$.
Moreover, it follows from \eqref{eq:critpointunten*} and the representation \eqref{eq:argsup*} that
\begin{align}\label{eq:Id1unten*}
   \hat g^*(\widehat{\nabla g(\xi)})-\hat f^*(\widehat{\nabla g(\xi)})=&   g^*(\nabla g(\xi))-f^*(\nabla g(\xi))\nonumber\\
   =& g^*(\nabla g(\xi))-f^*(\nabla f(\xi))+\xi^T(\nabla f(\xi)-\nabla g(\xi))=f(\xi)-g(\xi).
\end{align}

Conversely, let $\hat{\nu}$ be a critical point of $\hat{g}^*-\hat{f}^*$, i.\,e. $\nabla (\hat{g}^*(\hat{\nu}) -\hat{f}^*(\hat{\nu})) = 0$. Using Lemma \ref{lem:SD*} (both parts), we may therefore choose $\nabla f^*(\nu) \in \partial f^*(\nu)$ such that $\nabla (g^*(\nu) -f^*(\nu)) = 0$. As above, it follows that
\begin{align*}
\nabla f^* (\nabla g(\nabla g^*(\nu)))=\nabla f^*(\nu)=\nabla g^*(\nu).
\end{align*}
Applying $\nabla f$ to both sides, we obtain that $\nabla g^*(\nu)$ is also a critical point of $f-g$, and as above we may furthermore deduce that it satisfies
\begin{equation}\label{eq:Id2unten*}
    f(\nabla g^*(\nu)) - g(\nabla g^*(\nu)) = g^*(\nu)-f^*(\nu).
\end{equation}

Suppose in addition that $\xi$ is a maximizer such that $f(\xi)-g(\xi)=\sup \{f-g\}$. Then, by convex duality \eqref{eq:convexduality*} and \eqref{eq:Id1unten*}, we obtain that $\nabla g(\xi)$ is a maximizer of $g^*-f^*$. Conversely let $\nu$ be a maximizer of $g^*-f^*$, then by \eqref{eq:convexduality*} and \eqref{eq:Id2unten*} also $\nabla g^*(\nu)$ is a maximizer of $f-g$.
\end{proof}

Subsequently, no subdifferentials will appear and all $\nabla, \partial$ are classical derivatives.

\section{Finding the minimizer of $\phi$}\label{sec:Minimizer}

To start, we derive the appearance of the function
\begin{align}\label{eq:phi}
 \phi_N(\xi):=\phi(\xi):=\tfrac 1{2} \xi^T\mathcal A \xi-\sum_{k=1}^s \frac{\abs{S_k}}{N} \log\left(\sum_{c=1}^q \exp\left[\xi^T\mathcal A(e_k\otimes e_c) \right]\right)
 \end{align}
in Lemma \ref{lem:phi} as the exponential density for the smoothened distribution of the magnetization under the Gibbs measure.

\begin{proof}[Proof of Lemma \ref{lem:phi}]
As the Hubbard-Stratonovich transform from \cite{LS18, EW90} suggests, we begin to evaluate for any measurable set $U\in\mathcal B (\IR^{sq})$
\begin{align*}
 &\mu_N\circ ( \mathcal S^{\theta}m)^{-1}\ast\mathcal N (0,N( \mathcal S^{1-\theta}\mathcal A \mathcal S^{1-\theta})^{-1})(U)\\
 &=c_N\sum_{\omega\in\{1,\dots q\}^N} \int_U\exp\left[-\frac 1 2\left(x-\mathcal S^\theta m\right)^T\left( \frac{\mathcal S^{1-\theta}\mathcal A\mathcal S^{1-\theta}}{N}\right)\left(x-\mathcal S^\theta m\right)+\frac 1 2  \left(\frac {\mathcal S m}{\sqrt N}\right)^T\mathcal A \frac{\mathcal S m}{\sqrt N}\right] d^{sq}x
\\
&=c_N \int_U \exp\left[-\tfrac 12 x^T\left(\frac{\mathcal S^{1-\theta}\mathcal A\mathcal S^{1-\theta}}{N}\right) x\right]
\sum_{\omega\in\{1,\dots q\}^N}\exp\left[x^T\left(\frac{\mathcal S^{1-\theta}\mathcal A \mathcal S}{N}\right) m\right] d^{sq}x\\
&=c_N \int_U \exp\left[-\tfrac 12 x^T\left(\frac{\mathcal S^{1-\theta}\mathcal A\mathcal S^{1-\theta}}{N}\right) x\right]
\sum_{\omega_1=1}^q\cdots\sum_{\omega_N=1}^q\exp\left[\tfrac 1 { N} \sum_{k=1 }^s \sum_{c=1}^q \sum_{i\in S_k}x^T(\mathcal S^{1-\theta}\mathcal A)(e_k\otimes e_c)\eins_c(\omega_i)\right]d^{sq}x\\
&=c_N \int_U \exp\left[-\tfrac 12 x^T\left(\frac{\mathcal S^{1-\theta}\mathcal A\mathcal S^{1-\theta}}{N}\right) x\right]
\sum_{\omega_1=1}^q\cdots\sum_{\omega_N=1}^q\prod_{k=1}^s\prod_{i\in S_k}\exp\left[\tfrac 1 { N} x^T(\mathcal S^{1-\theta}\mathcal A)(e_k\otimes e_{\omega_i})\right]d^{sq}x\\
&=c_N \int_U \exp\left[-\tfrac 12 x^T\left(\frac{\mathcal S^{1-\theta}\mathcal A\mathcal S^{1-\theta}}{N}\right) x\right]
\prod_{k=1}^s\left(\sum_{c=1}^q \exp\left[\tfrac 1 { N} x^T(\mathcal S^{1-\theta}\mathcal A)(e_k\otimes e_c)\right]\right)^{\abs{S_k}}d^{sq}x,
\end{align*}
where in the last steps we basically used the fact that in each block $S_k$, the quantity is independent of the individual node $\omega_i$. Slightly rewriting, we therefore obtain
\begin{align*}
&\mu_N\circ ( \mathcal S^\theta m)^{-1}\ast\mathcal N (0,N(\mathcal S^{1-\theta}\mathcal A \mathcal S^{1-\theta})^{-1})(U)\\
&=c_N \int_U \exp\left[-\frac N2 \left( \frac{\mathcal S^{1-\theta}}{N}x \right)^T\mathcal A \left( \frac{\mathcal S^{1-\theta}}{N}x \right) +\sum_{k=1}^s \abs{S_k}\log\left(\sum_{c=1}^q \exp\left[\left( \frac{\mathcal S^{1-\theta}}{N}x \right)^T\mathcal A(e_k\otimes e_c) \right]\right)\right]d^{sq}x.
\end{align*}
For simplicity, we write $\phi$ as presented in \eqref{eq:phi} and shift by an arbitrary $v\in\IR ^{sq}$ to obtain
\begin{align*}
     &\mu_N\circ \left(\mathcal S^\theta(m-v)\right)^{-1}\ast\mathcal N \left(0,N( \mathcal S^{1-\theta}\mathcal A \mathcal S^{1-\theta})^{-1}\right)(d^{sq}x)\nonumber\\
 &=c_N \exp\left[- N \phi\left(\left(\frac{\mathcal S}{ N}\right)^{1-\theta}\frac x { N^{\theta}}+\frac {\mathcal S}{N}v \right)\right]d^{sq}x,
\end{align*}
as claimed.
\end{proof}

The central aim of this section is to find the minimizer of $\phi$. To see that the set of minimizers is non-empty, note that there exists an $R>0$ such that
\begin{align}\label{eq:defR}
 \phi(\xi)\ge\tfrac 13\xi^T\mathcal A \xi\quad\text{for all }\norm{\xi}>R
\end{align}
for $N$ sufficiently large, i.\,e.\ $\varphi$ behaves at least quadratically for $\lVert \xi \rVert$ large. Indeed, this follows from considering the limit $\norm\xi\to\infty$ and $\mathcal A$ being positive definite, i.\,e.\ $\tfrac 1 6 \xi^T\mathcal A \xi\ge \lambda\norm{\xi}^2$ for some $\lambda>0$. 

It turns out that finding the minimizer in the case of (asymptotically) equal block sizes is fairly easily dealt with by the convex duality results from the previous section and \cite{Prepreprint}.

\begin{lem}\label{lem:Minimizer}
The minimum $\phi^*=\inf_{\xi\in\IR^ {sq}}\phi(\xi)$ of $\phi$ is attained at
\begin{align}\label{setminimizers}
    \xi^*\in\Big\{\xi\in (0,1) ^{sq}: \sum_{c=1}^{q}\xi_{k,c}=\frac{\abs{S_k}}N\text{ for all }k=1,\dots ,s\Big\} = C 
\end{align}
with $C$ as in \eqref{DasIstC}. Under the conditions of Theorem \ref{thm:CLT}, the minimum is uniquely attained at $\xi^* = \frac 1 {q s}\mathbf 1 _{sq}$ in the case of equal block sizes. If the block sizes are asymptotically equal only, then for any (big) $R>0$ and any (small) $r>0$ to be chosen later, any minimizer $\xi_N^*$ of $\phi_N$ in $B_R(0)$ satisfies $\xi^*_N\in B_r(\frac 1 {q s}\mathbf 1 _{sq})$ for $N$ large enough. 
\end{lem}

\begin{proof}
To see the first claim, first note that
\begin{align}\label{eq:NablaPhi}
    \nabla \phi(\xi)=\mathcal A\xi- \sum_{k=1}^s\frac{\abs{S_k}}{N} \frac{\sum_{c=1}^q\mathcal A(e_k\otimes e_c)\exp(\xi ^T\mathcal A(e_k\otimes e_c))}{\sum_{c=1}^q\exp(\xi^T\mathcal A(e_k\otimes e_c))}=0,
\end{align}
so that we must have $\xi_{k,c} > 0$. Moreover, we may plug the upper identity into the $s$ equations $\sum_{c=1}^q \partial_{k,c} \phi(\xi) = 0$, which immediately leads to the result.

In order to find the supremum of $-\phi$, note that by Lemma \ref{convexdualityfg*},
\begin{align*}
\sup_{\xi\in\IR^{sq}} \{-\phi(\xi)\}&= \sup_{\tilde\xi=\mathcal A \xi}\left\{\sum_{k=1}^s\frac{\abs{S_k}}{N}\log\sum_{c=1}^q\exp(\tilde \xi_{k,c})- \frac12\tilde\xi  ^T\mathcal A ^{-1}\tilde\xi \right\}\\
&=\sup_{\tilde\xi \in\IR^{sq}}\left\{ f(\tilde\xi )-g(\tilde\xi )\right\}\\
&=\sup_{\nu\in\mathrm{dom}(f^*)}\{g^*(\nu)-f^*(\nu)\},
\end{align*}
and if $\nu^*$ is a maximizer of $g^*-f^*$, then $\tilde\xi^*=\nabla g^*(\nu^*)$ maximizes $f-g$. Since $\nabla g^*(\nu) = \mathcal A \nu$, it follows that $\nu^*$ is a minimizer of $\phi$.

Therefore, it remains to check the maximizers of $g^*-f^*$. Here, using \eqref{convexconjugates*}, we have
\[
\sup_{\nu \in \mathrm{dom}(f^*)} \{g^*(\nu) - f^*(\nu)\} = \sup_{\nu\in C}\left\{\frac{1}{2}\nu^T\mathcal A \nu-\nu^T\log \nu\right\} - \mathcal H.
\]
In \cite[\S 4 \& (4.1)]{Prepreprint}, it has been shown that for equal block sizes $|S_k|=N/s$, the supremum on the right-hand side is attained at a unique point $\nu^*=\frac 1{ sq} \mathbf 1_{sq}$. For asymptotically equal block sizes, the function $\phi_N$ converges uniformly on compact sets to
\[
\phi_\infty(\xi) :=\tfrac 1{2} \xi^T\mathcal A \xi-\frac{1}{s}\sum_{k=1}^s \log\left(\sum_{c=1}^q \exp\left[\xi^T\mathcal A(e_k\otimes e_c) \right]\right),
\]
i.\,e.\ $\phi$ from \eqref{eq:phi} for (non-asymptotically) equal block sizes. Take $R > 0$ such that $\frac 1 {sq}\mathbf 1_{sq}\in B_R(0)$. Then, for any small $\varepsilon > 0$, the minimum $\phi_N^*$ of $\phi_N$ must satisfy $\phi^*-\varepsilon < \phi_N^* < \phi^*+\varepsilon$ by uniform convergence. In particular, it follows that if $\xi_N^*$ is a minimizer of $\phi_N$, then by continuity, for any small $r > 0$, $\xi^*_N\in B_r(\frac 1{ sq} \mathbf 1_{sq})$ for any $N$ sufficiently large.
\end{proof}

Let us now turn to the case of general block sizes, for which we will study $\phi$ in more detail. This analysis is inspired by \cite{EW90}, which in turn relies on the ideas of \cite{KeSchon89}. In our case of $s>1$ blocks however, we need to replace a number of arguments that cannot be generalized directly, like monotonicity (in $\beta$ and $t$ defined below), real convex functions having at most two zeros or applying the logarithm.

According to \eqref{eq:NablaPhi}, any minimizer satisfies the critical equation
\begin{align}\label{eq:criteqn}
\xi_{k,c}:=(e_k\otimes e_c)^T\xi= \frac{\abs{S_k}\exp(\xi ^T(Ae_k\otimes e_c))}{N\sum_{\tilde c=1}^q\exp(\xi ^T(Ae_k\otimes e_{\tilde c}))}.
\end{align}
 For simplicity we write $\xi_{\bullet,c}=(I_s\otimes e_c^T)\xi\in\IR^s$ for the vector of a color $c$ in different blocks and $\xi_{k,\bullet}=(e_k^T\otimes I_q)\xi\in\IR^q$ for the vector of colors in a block $k$. Using the componentwise exponential and abbreviating $w_k=\frac{\exp(\xi^T(Ae_k\otimes I_q))}{\sum_{\tilde c=1}^q\exp(\xi ^T(Ae_k\otimes e_{\tilde c}))}\in\IR^q$, we may write the Hessian of $\phi$ as the block matrix
\begin{align}\label{eq:Hessian}
 H_\phi(\xi)&=A\otimes I_q-\sum_{k=1}^s\left(\frac{AS}{N}e_k(Ae_k)^T \right)\otimes\left(\operatorname{diag}(w_k)-w_kw_k^T\right)\\
 &=\mathcal A\left[I_{sq}-\sum_{k=1}^s\left(e_ke_k^T \right)\otimes\left(\operatorname{diag}(\tfrac{\abs{S_k}}Nw_k)-\tfrac{\abs{S_k}}Nw_kw_k^T\right)\mathcal A\right].\nonumber
\end{align}
Also note that by \eqref{eq:criteqn}, any minimizer $\xi$ satisfies $\xi_{k,\bullet}=\tfrac{\abs{S_k}}Nw_k$.

Rearranging \eqref{eq:criteqn} and applying the logarithm, we obtain for all $k=1,\dots,s$ that
\begin{align}\label{eq:KLS4.3}
& \xi_{\bullet,c}^TA e_k-\log(\xi_{\bullet,c}^T e_k)=\log\left(\sum_{\tilde c=1}^q\exp(\xi^T_{\bullet,\tilde c}Ae_k)\right)-\log\left(\frac{\abs{S_k}}{N}\right)
\end{align}
is independent of $c=1,\dots,q$. Multiplying the same equation with $\xi_{k,c}$ and summing over $k$ and $c$ implies that for any critical point $\xi^*$ the value of $\phi$ can also be represented as a relative entropy
\begin{align*}
    \phi(\xi^*)=-\tfrac 1 2 \sum_{k,c}\xi_{k,c}\log\left(\frac{\sum_{\tilde c}^q \exp\big(\xi_{\bullet,\tilde c}^TAe_k\big)}{\xi_{k,c}}\right)-\tfrac 1 2\mathcal H(|S_\bullet|/N).
\end{align*}

Let us collect some important properties that will help to classify and parametrize minimizers, generalizing \cite[Lemmas 4.3 \& 4.4]{Prepreprint}.

 \begin{lem}\label{lem:ordered+atmost2} Any minimizer $\xi$ of $\phi$ satisfies
 \begin{enumerate}
     \item The color-coordinates may be ordered decreasingly, i.e. $\xi_{k,1}\ge\dots\ge\xi_{k,q}$ for all $k$.
      \item If for some $c=1,\dots,q$ we have $\xi_{k,c}=\xi_{k,c+1}$ for some $k$, then it holds for all $k$.
     \item There are at most two different entries in $\xi _{k,\bullet}$, for any $k=1,\dots,s$.
 \end{enumerate}
 \end{lem}
 
 This is the only proof that uses the assumption of $A$ having positive entries.
 
  \begin{proof}
 (1)
 As we have seen in the proof of Lemma \ref{lem:Minimizer}, we have of $\phi$
\[
\sup_{\xi\in\IR^{sq}} \{-\phi(\xi)\} = \sup_{\nu\in C}\left\{\frac{1}{2}\nu^T\mathcal A \nu-\nu^T\log \nu\right\} - \mathcal H.
\]
Hence, by convex duality $\xi $ is a minimizer if and only if $\nu =\xi $ is a minimizer of $\xi^T\log\xi-\tfrac 1 2 \xi^T\mathcal A \xi$. Let us assume that the color coordinates are ordered $\xi_{k,1}\ge\dots\ge\xi_{k,q}$ for all $k$ and show that any tuple $\sigma$ of $s$ permutations $\sigma_k\in\mathbb S ^q$, given by $(e_k\otimes e_c)^T\xi_\sigma=\xi_{k,\sigma_k(c)}$ would only increase the value of $\phi$. Indeed, for any $j,k=1,\dots,s$ the rearrangement inequality \cite[Theorem 368]{Ineq} states that $\sum_{c=1}^q\xi_{j,c}\xi_{k,c}\ge\sum_{c=1}^q\xi_{j,\sigma_j(c)}\xi_{k,\sigma_k(c)} $, hence
\begin{align}\label{eq:ordering}
    \xi^T\log\xi-\frac 1 2 \xi^T\mathcal A \xi=\sum_{k=1} ^s\sum_{c=1}^q\xi_{k,c}\log\xi_{k,c}-\frac 1 2\sum_{j,k=1}^sA_{j,k}\sum_{c=1}^q\xi_{j,c}\xi_{k,c}\le \xi_\sigma^T\log\xi_\sigma-\frac 1 2 \xi_\sigma^T\mathcal A \xi_\sigma.
\end{align}
Here, we used that $\xi^T\log\xi$ is invariant under permutation of the vector and that $A_{j,k}> 0$. Note that the rearrangement inequality is strict if two different indices are reordered differently, i.\,e.\ the value of \eqref{eq:ordering} may only increase. Thus, if there exists $j\neq k$ and $\tilde c <c$ with different orderings $\xi_{j,\tilde c}<\xi_{j,c}$ and $\xi_{k,\tilde c}>\xi_{k,c}$, then it is not a minimum of \eqref{eq:ordering} and not a minimum of $\phi$.

 (2) Assume $\xi _{k,c}=\xi _{k, c+1}$ for some $k=1,\dots,s $ and some $c$ and consider the difference of \eqref{eq:KLS4.3} for $k,c$ and for $k,c+1$, which is equivalent to
 \begin{align*}
     (\xi_{\bullet,c}-\xi_{\bullet,c+1})^TAe_k-\log(\xi_{k,c}/\xi_{k,c+1})=0\ \Leftrightarrow\ \sum_{j=1}^sA_{j,k}\xi_{j,c}=\sum_{j=1}^sA_{j,k}\xi_{j,c+1}.
 \end{align*}
 On the other hand by (1), we have $\xi_{k,c}\ge\xi_{k, c+1}$ for all $k$ and equality follows from $A_{j,k}> 0$.
 
 (3) Suppose there exists $c_1<c_2<c_3$ with $\xi _{k,c_1}>\xi _{k,c_2}>\xi _{k,c_3}$ for some $k=1,\dots,s$ and hence for all $k$ due to (2). Again, from \eqref{eq:KLS4.3} it follows
  \begin{align*}
       \frac{(\xi_{\bullet,c_1}-\xi_{\bullet,c_2})^TAe_k}{\xi_{k,c_1}-\xi_{k,c_2}}=\frac{\log(\xi_{k,c_1})-\log(\xi_{k,c_2})}{\xi_{k,c_1}-\xi_{k,c_2}}<\frac{\log(\xi_{k,c_1})-\log(\xi_{k,c_3})}{\xi_{k,c_1}-\xi_{k,c_3}}=     \frac{(\xi_{\bullet,c_1}-\xi_{\bullet,c_3})^TAe_k}{\xi_{k,c_1}-\xi_{k,c_3}},
 \end{align*}
 where the inequality is nothing but the concavity of the logarithm. Multiplying with $(\xi_{k,c_1}-\xi_{k,c_2})$ and $(\xi_{k,c_1}-\xi_{k,c_3})$ and then summing over $k=1,\dots,s $ implies 
 \begin{align*}
     (\xi_{\bullet,c_1}-\xi_{\bullet,c_2})^TA(\xi_{\bullet,c_1}-\xi_{\bullet,c_3})<(\xi_{\bullet,c_1}-\xi_{\bullet,c_3})^TA(\xi_{\bullet,c_1}-\xi_{\bullet,c_2}),
 \end{align*}
 which is a contradiction since $A$ is symmetric.
 \end{proof}
 
Rewriting Lemma \ref{lem:ordered+atmost2}, we may search the minimizer $\xi^*$ of $\phi$ among the vectors $\xi$ for which there exists $l \in \{0, 1, \ldots, q-1\}$ such that for any $k = 1, \ldots, s$
\[
\xi_{k,1} = \ldots = \xi_{k,l} > \xi_{k,l+1} = \ldots = \xi_{k,q}.
\]
Recalling \eqref{setminimizers}, for each $k$ there exists $t_k \in [0,1]$ such that the two possible values are given by $|S_k|t_k/(Nl) + |S_k|(1-t_k)/(Nq)$ and $|S_k|(1-t_k)/(Nq)$, i.\,e.\ any minimizer is of the form
\begin{align}\label{formofmin}
 \xi = \xi(t) =\sum_{c=1}^l\frac{St}{Nl}\otimes e_c+ \frac SN(\mathbf 1_s-t)\otimes\tfrac{1}{q}\mathbf 1_q
\end{align}
for some $t \in [0,1]^s$. Next, we show that $l\le 1$, i.\,e.\ only the first coordinate may have a different value from the remaining ones.

\begin{lem}\label{lem:XiDarst}
Any minimizer of $\phi$ is given by
 \begin{align}\label{eq:XiDarst}
 \xi(t)=\frac{S}{N}t\otimes e_1+ \frac SN(\mathbf 1_s-t)\otimes\tfrac{1}{q}\mathbf 1_q
 \end{align}
 for some $t\in [0,1]^s$.
 \end{lem}
 
This replaces \cite[Equation (2.8)]{EW90}, where $t\in[0,1]$. Surprisingly, it is not sufficient to consider values $t=\tilde t \mathbf 1_s$, $\tilde t\in [0,1]$, in \eqref{eq:XiDarst} in order to follow the lines of \cite{EW90}. It will be apparent from Figure \ref{fig:phi} below that minima might appear off the diagonal.
In the sequel, we will use the Loewner order of matrices $X>_LY$ (or $X\ge_L Y$) if $X-Y$ is positive (semi-) definite.

 \begin{proof}
Let us first consider the case $\operatorname{diag}(\xi_{\bullet,1})\le_L A^{-1}$, which is equivalent to $A\le_L\operatorname{diag}(\xi_{\bullet,1})^{-1}$. 
Let us view the left hand side of \eqref{eq:KLS4.3} as the component $g_k(\xi_{\bullet,c})$ for a function $g:\IR^s\to\IR^s$. Then by our assumption, $Dg(x)=A-\operatorname{diag}(x)^{-1}<_L0$ for all $x\in\IR^s$ such that $\operatorname{diag}(x)<_L\operatorname{diag}(\xi_{\bullet,1})$. Hence $g(\xi_{\bullet,c})=g(\xi_{\bullet,1})$ readily implies $\xi_{\bullet,c}=\xi_{\bullet,1}$, i.e. $\xi=\xi(0)$.

On the other hand, let us now assume $\operatorname{diag}(\xi_{\bullet,1})\not\leq_L A^{-1}$. Since $A$ is symmetric and positive definite, this is equivalent to $A-A\operatorname{diag}(\xi_{\bullet,1})A\not\ge_L0$, that is $u^T(A-A\operatorname{diag}(\xi_{\bullet,1})A)u<0$ for some $u\in \IR^s$. Now assume $l>1$ in \eqref{formofmin}. In this case, if $\xi$ is a minimizer of $\phi$, then $x=0$ is a minimum of the function $h(x)=\phi(\xi+xu\otimes(e_1-e_l))$. Contrarily its second derivative is explicitly given by
\begin{align*}
 h''(0)&=(u\otimes (e_1-e_l))^TH_\phi(\xi)(u\otimes (e_1-e_l))\\
 &=2u^TAu-\sum_{k=1}^s\left(u^TAe_k\right)^2(e_1-e_l)^T\left(\operatorname{diag}(\xi_{k,\bullet})-\tfrac{N}{\abs{S_k}}\xi_{k,\bullet}\xi_{k,\bullet}^T\right)(e_1-e_l)\\
 &=2u^TAu-\sum_{k=1}^s \left(u^TAe_k\right)^2\left(\xi_{k,1}+\xi_{k,l}-\tfrac{N}{\abs{S_k}}(\xi_{k,1}-\xi_{k,l})^2\right)\\
  &=2u^T\left(A-A\operatorname{diag}(\xi_{\bullet,1})A\right)u<0.
\end{align*}
Therefore $\xi$ cannot be a minimizer of $\phi$ and the contradiction proves $l\le1$.
 \end{proof}
 
 \begin{figure}[ht]
 \includegraphics[width=.33\textwidth]{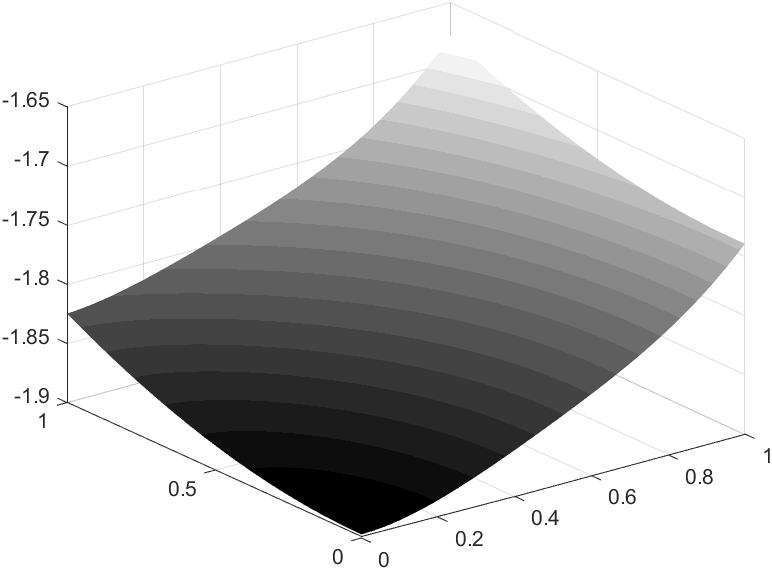}\includegraphics[width=.33\textwidth]{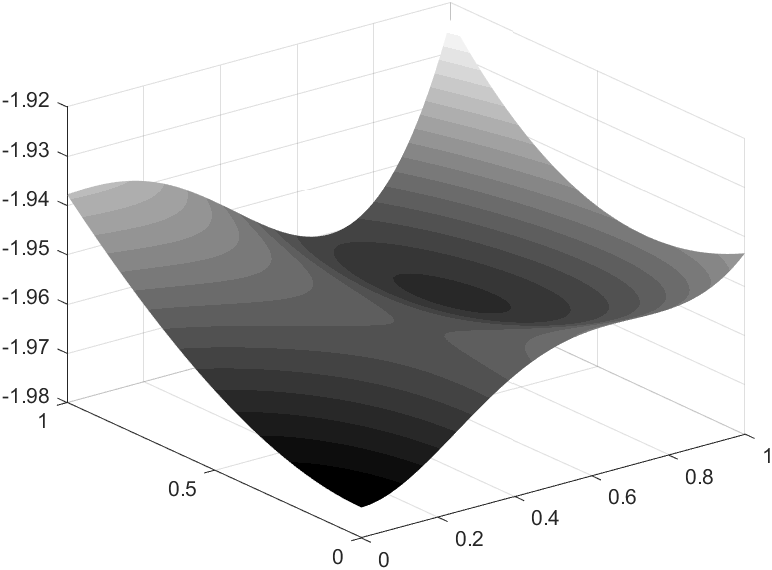}\includegraphics[width=.33\textwidth]{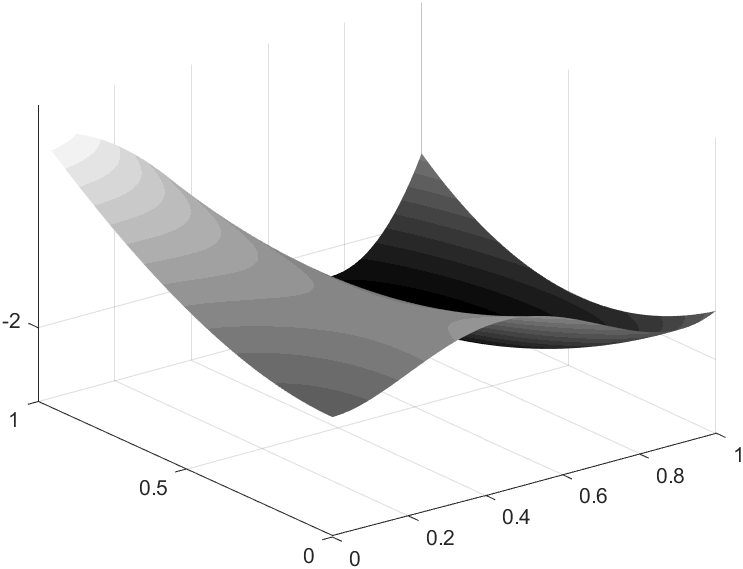}
   \caption{The function $[0,1]^2\ni t\mapsto\phi(\xi(t))$, for $q=5$ colors\protect\footnotemark, $s=2$ blocks of size $\gamma=(1/4,3/4)$ and different regimes of the norm of the interaction matrix: Smaller than $4\tfrac{q-1}q=3.2$ (left), bigger than $\zeta_q=\tfrac 83 \log 4\approx 3.7$ (right) and in between (middle). The left graph clearly shows that for $\lVert\sqrt\Gamma \mathcal A\sqrt\Gamma\rVert=3.1<3.2$, the (global) minimum is attained at $t=0$. If we increase the inverse temperature to $3.2<\lVert\sqrt\Gamma \mathcal A\sqrt\Gamma\rVert=3.65< 3.7$ (center), another local minimum appears at $t\neq 0$. If $\lVert\sqrt\Gamma \mathcal A\sqrt\Gamma\rVert=3.8>\zeta_q$ (right), then the new local minimum becomes global and a CLT around $\tfrac 1 q \mathbf 1_{sq}$ cannot hold.}\label{fig:phi}
   \end{figure}
  \footnotetext{Colors of the Potts model, not in the picture.}

Next, we study the minimizer of $\phi(\xi(t))$ in the regime where the interaction matrix has small norm as illustrated in Figure \ref{fig:phi} on the left. Figure \ref{fig:phi} also exposes the difficulty of the intermediate regime, which is not covered by our results for inhomogeneous block sizes. 

\begin{prop}\label{prop:Minimizer}
If $\norm{\sqrt{\tfrac SN}A\sqrt{\tfrac SN}}_2<4\frac{q-1}{q}$, then $\xi^*= \tfrac{\mathcal S}{qN} \mathbf 1 _{sq}$ is the unique minimizer of $\phi$.
\end{prop}

\begin{proof}
By Lemma \ref{lem:XiDarst} it remains to study the minima of $\Phi(t)=\phi(\xi(t))$ as a function of $t\in[0,1]^s$. We have
\begin{align*}
\Phi (t)=& \tfrac 1 2 \xi(t)^T(A\otimes  I_q)\xi(t)-\sum_{k=1}^s\frac{\abs{S_k}} N \log\left[\sum_{c=1}^q\exp(\xi(t)^T(Ae_k\otimes e_c))\right]\\
=&\tfrac 1 2 t^T \frac{SAS}{N^2}t+t^T\frac{SAS}{N^2}(\mathbf 1_s-t)\tfrac 1 q+\tfrac 1 2(\mathbf 1_s-t)^T\frac{SAS}{N^2}(\mathbf 1_s-t)\tfrac 1 q\\
&-\sum_{k=1}^s\frac{\abs{S_k}} N \log\left[ \exp\Big(\big(\tfrac 1 q\mathbf 1_s+\tfrac{q-1}{q} t\big)^T\tfrac{SA}Ne_k\Big)+(q-1)\exp\Big(\tfrac 1 q\big(\mathbf 1_s-t)^T\tfrac{SA}{N}e_k\big)\Big)  \right]\\
=&\frac{q-1}{2q} t^T\frac{SAS}{N^2}t +\frac{1}{2q}\mathbf 1_s^T \frac{SAS}{N^2}\mathbf 1_s\\
&-\sum_{k=1}^s\frac{\abs{S_k}} N \left( \tfrac 1 q\mathbf 1^T_s \frac {SA}N e_k+ \tfrac{q-1}q t^T\tfrac{SA}{N}e_k+\log\left[1+(q-1)\exp\big(-t^T\tfrac{SA}{N}e_k\big)\right]\right)\\
=&\frac{q-1}{2q} t^T\frac{SAS}{N^2}t -\frac{1}{2q}\mathbf 1_s^T \frac{SAS}{N^2}\mathbf 1_s-\frac{q-1}q t^T\frac{SAS}{N^2}\mathbf 1_s-\sum_{k=1}^s\frac{\abs{S_k}} N \log\left[1+(q-1)\exp\big(-t^T\tfrac{SA}{N}e_k\big)\right].
\end{align*}
In this form, we easily obtain the derivative
\begin{align*}
    \nabla \Phi(t)=&\frac{q-1}q\frac{SAS}{N^2}(t-\mathbf 1_s)+\sum_{k=1}^s\frac{\abs{S_k}}N \frac{(q-1)\tfrac{SA}Ne_k\exp\big(-t^T\tfrac{SA}Ne_k\big) }{1+(q-1)\exp\big(-t^T\tfrac{SA}Ne_k\big)}\\
    =&\sum_{k=1}^s\frac{q-1}{q}\frac{SAS}{N^2}e_k\left(\frac{1+(q-1)t_k-(1-t_k)\exp\big(t^T\tfrac{SA}Ne_k\big) }{(q-1)+\exp\big(t^T\tfrac{SA}Ne_k\big)}\right)\\
    =&\frac{q-1}{q}\frac{SAS}{N^2}\left(\sum_{k=1}^s\frac{q e_k}{(q-1)+\exp\big(t^T\tfrac{SA}Ne_k\big)}-\big(\mathbf1_s-t\big)\right)
\end{align*}
Similar to \eqref{eq:criteqn}, a critical point satisfies $\nabla\Phi=0$ and it follows\footnote{Notice that $\nabla \Phi$ is the generalization of the derivative $A_\beta'$ in \cite[Equation (5.4)]{EW90} and the following equation \eqref{eq:crit_t} corresponds to \cite[Equation (2.9)]{EW90}.}
\begin{align}\label{eq:crit_t}
    t=\mathbf1_s -\sum_{k=1}^s\frac{q e_k}{(q-1)+\exp\Big(t^T\tfrac{SA}Ne_k\Big)}=:h(t),
\end{align}
hence we are looking for a fixed point of the auxiliary function $h$. The most obvious candidate is the fixed point $t=0$, which we shall prove to be unique. The Jacobian of $h$ is given by
\begin{align*}
    Dh(t)= \sum_{k=1}^s\frac{q e_k\big(\tfrac{SA}Ne_k\big)^T\exp\big(t^T\tfrac{SA}Ne_k\big)} {\Big((q-1)+\exp\big(t^T\tfrac{SA}Ne_k\big)\Big)^2}
    = \sum_{k=1}^s\frac{q e_ke_k^T \exp\big(t^T\tfrac{SA}Ne_k\big)} {\Big((q-1)+\exp\big(t^T\tfrac{SA}Ne_k\big)\Big)^2}A\frac SN.
\end{align*}

The fixed point $t=0$ is unique for $h$ if and only if it is the unique fixed point for $\tilde h(t)=(\tfrac SN)^{1/2}h((\tfrac SN)^{-1/2}t)$. The latter follows from the Banach fixed point theorem, if the spectral norm of the Jacobian 
\begin{align*}
       D\tilde h (t) =(\tfrac SN)^{1/2}(Dh)((\tfrac SN)^{-1/2}t)(S/N)^{-1/2}= \sum_{k=1}^s\frac{q e_ke_k^T \exp\Big(t^T\sqrt{\tfrac{S}N}Ae_k\Big)} {\Big((q-1)+\exp\Big(t^T\sqrt{\tfrac{S}N}Ae_k\Big)\Big)^2}\sqrt{\frac{S}N}A\sqrt{\frac{S}N}
\end{align*}
 is smaller than 1 for all $t$. A simple analysis shows that the function $x\mapsto q\exp(x)/(q-1+\exp(x))^2$ has its maximum at $x=\log(q-1)$ with value $q/(4(q-1))$.\footnote{For some choice of $S$ and $A$ this maximum can also be attained.} Therefore, the spectral norm is bounded by
\begin{align*}
    \norm{D\tilde h(t)}_2\le\frac{q}{4(q-1)}\norm{\sqrt{S}A\sqrt{S}/N}_2<1,
\end{align*}
where the first inequality holds since the first matrix is diagonal and the second inequality follows from our assumption. The Banach fixed point theorem implies that the fixed point $t=0$ is unique and the claim follows from Lemma \ref{lem:XiDarst}.
\end{proof}

Note that the vector $\xi^*=\tfrac{\mathcal S}{qN}\mathbf 1_{sq}$ is always a critical point. Moreover, it follows from Lemma \ref{lem:Minimizer} and Proposition \ref{prop:Minimizer} that for equal block sizes or in the situation of Theorem \ref{thm:CLTgeneral}, it is a global minimizer.

Next, we verify that the Hessian of $\phi$ is positive definite at $\xi^*$ as well as in some neighbourhood which can be chosen uniformly for $N$ sufficiently large, cf. \cite[Lemma 3.4a]{EW90}.

\begin{lem}\label{lem:posdef}
If $\lVert \sqrt \Gamma \mathcal A \sqrt \Gamma\rVert_2<q$, then for $N$ sufficiently large there exists a neighbourhood $B_r(\xi^*)$ of $\xi^*=\tfrac{\mathcal S}{qN}\mathbf 1_{sq}$, where the Hessian of $\phi$ is positive definite, i.e. there exists $r,\lambda_{\min}>0$ such that
\begin{align}\label{eq:posdef}
\xi^TH_\phi(\tilde\xi)\xi>\tfrac 1 2 \lambda_{\min}\norm\xi^2 
\end{align}
for all $\xi\in\IR ^{sq},\tilde\xi\in B_r(\xi^*)$. 
\end{lem}

In particular, for asymptotically equal block sizes, $\xi^*$ converges to $\frac 1 {q s}\mathbf 1 _{sq}$ and any minimizer $\xi_N^*$ lies in a neighborhood of $\frac 1 {q s}\mathbf 1 _{sq}$ for $N$ large enough by Lemma \ref{lem:Minimizer}. Combining these facts with Lemma \ref{lem:posdef}, we see that for $N$ large enough, $\xi^*$ and the so far unspecified vectors $\xi_N^*$ from Lemma \ref{lem:Minimizer} are minimizers which all lie in an area of positive definite Hessian. This readily implies uniqueness of the minimizer in some ball $B_R(0)$. Choosing $R$ sufficiently large (cf.\ \eqref{eq:defR}) it follows that the unique minimizer of $\phi$ is indeed given by $\xi^*=\tfrac{\mathcal S}{qN}\mathbf 1_{sq}$ under the assumptions from Theorem \ref{thm:CLT}.

\begin{proof}
The Hessian is given by \eqref{eq:Hessian}, where we already saw that at the critical point $\xi^*=\tfrac{\mathcal S}{qN}\mathbf 1_{sq}$ we may plug in the critical equation \eqref{eq:criteqn}. Thus, in terms of the Loewner order, we have the equivalence
\begin{align}
 H_\phi(\xi^*)&=\mathcal A\left[ I_{sq}-\sum_{k=1}^s(e_k e_k^T)\otimes\left(\operatorname{diag}(\xi^*_{k,\bullet}) -\tfrac{N}{\abs{S_k}}\xi^*_{k,\bullet}{\xi^*_{k,\bullet}}^T\right)\mathcal A\right]\nonumber\\
&=A\otimes I_{q}-\frac{ASA}{q N}\otimes(I_q-\tfrac 1 q\mathbf 1_{q\times q})>_L0\label{eq:Hessianexplicit}\\
&\Leftrightarrow\frac{\sqrt SA\sqrt S}N\otimes I_{q}-\frac{\sqrt S ASA\sqrt S}{q N^2}\otimes(I_q-\tfrac 1 q\mathbf 1_{q\times q})>_L0\nonumber\\
&\Leftrightarrow\frac{\sqrt S A\sqrt S}{q N}\otimes(I_q-\tfrac 1 q\mathbf 1_{q\times q})<_L I_s\otimes I_q.\label{eq:LoewnerEquivalence}
\end{align}
Similarly to the proof of Lemma \ref{lem:XiDarst}, these equivalences hold since all the occurring matrices are symmetric and positive definite. For example the eigenvectors of $A_{\alpha,\beta}$ are $v_1=\mathbf 1_s$ with eigenvalue $\lambda_1=\beta+(s-1)\alpha$ and $v_k=e_k-e_{k-1}$ for $1<k\le s$ having eigenvalue $\lambda_k=\beta-\alpha>0$. In the same notation, they are also the eigenvectors of $I_q-\tfrac 1 q\mathbf 1 _{q\times q}$ (corresponding to $s\leftrightarrow q$, $\beta=(q-1)/q$ and $\alpha=-1/q$) to the eigenvalues $0$ and $1$.
In other words we have $0\le_L I_q-\tfrac 1 q\mathbf 1 _{q\times q}\le_L I_q$. By \cite[Theorem 2.5]{B89} it holds that the Kronecker product preserves the Loewner order. Applied to our setting, it remains to show $\frac{\sqrt S A\sqrt S}{Nq}<_LI_s$, since then \eqref{eq:LoewnerEquivalence} follows from the Kronecker product with $I_q-\tfrac 1 q\mathbf 1 _{q\times q}\le_L I_q$. Recall our assumption $\lVert\frac{\sqrt S A\sqrt S}{N}\rVert_2\le q$ hence $\sqrt SA\sqrt S/N<_LqI_s$ holds.
Thus we have shown $H_\phi(\xi^\ast)$ is positive definite with minimal eigenvalue $\lambda_{\min} >0$. By continuity of $H_\phi$, there exists $r>0$ such that \eqref{eq:posdef} follows.
\end{proof}


\section{Proofs of the theorems}\label{sec:proofs}

We will prove the Central Limit Theorems via convergence of the moment-generating function.\footnote{The tail bound is even more immediate if one uses characteristic functions instead, but we will need moment-generating functions for large deviation considerations anyway.}
\begin{proof}[Proof of Theorem \ref{thm:CLT} and Theorem \ref{thm:CLTgeneral}]
According to the previous section, we assume that $\xi^\ast(=\frac{\mathcal S }{qN}\mathbf 1_{sq})$ is the unique minimizer of $\phi$, and the Hessian $H_\phi$ is positive definite in some neighborhood. Applying Lemma \ref{lem:phi} for $\theta=1/2$,  $v=\big(\frac {\mathcal S}N\big)^{-1} \xi^*(=\tfrac 1 q\mathbf 1 _{sq})$ and $Z\sim \mathcal N (0,N(\sqrt{ \mathcal S}\mathcal A\sqrt{ \mathcal S})^{-1})$ implies for all $t\in\IR ^{sq}$
\begin{align}
 &\int \exp\left[ t^T\big(Z+ \sqrt{\mathcal S}(m-v)\big)\right]d\IP\nonumber\\
 =&c_N \int \exp\left[- N \phi\left(\xi ^*+\sqrt{\frac{\mathcal S}{ N}}\frac x {\sqrt N}\right)+t^Tx\right]d^{sq}x\nonumber\\
 =&c_Ne^{-N\phi^*} \int \exp\left[- N\big(\phi(\xi)-\phi^*\big)+t^TN\sqrt {\mathcal S ^{-1}}(\xi-\xi^*)\right]\abs{\det(\sqrt{\mathcal S}/N) }^{-1}d^{sq}\xi,\label{eq:MEF}
\end{align}
where we used a change of variables $\xi=\xi ^*+\sqrt{\mathcal S/ N}x /{\sqrt N}$ and complemented the minimum $\phi^*=\phi(\xi^*)$ of $\phi$ in the exponent. Note that the Jacobian is of order $\sim N^{sq/2}$. For $r>0$ from Lemma \ref{lem:posdef}, we discuss the integral over the areas $B_r(\xi^*)$ and its complement separately. The latter can be bounded as follows. Since $\xi\in B^c_r(\xi^\ast)$ is not the global minimizer of $\phi$, we have $\phi(\xi)\ge\phi^*-\delta$ for some $\delta>0$. We choose $R>0$ from \eqref{eq:defR} and possibly enlarge it so that $\norm t \le R$ in order to estimate
\begin{align}\label{eq:Flanken}
&\int_{B^c_r(\xi^*)} \exp\left[- N\big(\phi(\xi)-\phi^*\big)+t^TN\sqrt {\mathcal S ^{-1}}(\xi-\xi^*)\right]d^{sq}\xi\\
&\le \int_{B^c_r(\xi^*)\cap B_R(0)} \exp\left[-N\delta+c\sqrt N R^2\right]d^{sq}\xi+e^{N\phi^\ast}\int_{B^c_r(\xi^*)\cap B^c_R(0)} \exp\left[-\frac N 3 \xi^T\mathcal A\xi+ c \sqrt N\norm \xi ^2\right]d^{sq}\xi\nonumber\\
&\le Ce^{-\epsilon N}.\nonumber
\end{align}
for some $\epsilon,c,C>0$ and $N$ sufficiently large. The last step follows from $\phi^*\le \phi(0)=-s\log q<0$ and $\mathcal A$ being positively definite. We also see that $\epsilon$ does depend on $\mathcal A$ (or more precisely on $\abs{\beta-\alpha}$ for $A=A_{\alpha,\beta}$), whereas $C$ is independent of it.

The major contribution comes from the local part of $\xi\in B_r(\xi^*)$, where we use a second order Taylor approximation with Lagrange remainder, i.e. there exists an intermediate point $\tilde\xi\in B_r(\xi^*)$ between $\xi$ and $\xi ^*$ such that
\begin{align}\label{eq:Taylor}
\phi(\xi)=\phi(\xi^*)+(\xi-\xi^*)^T\nabla\phi(\xi^*)+\tfrac 1 2 (\xi-\xi^*) ^T H_\phi (\tilde\xi)(\xi-\xi^*).
\end{align}
Note that Lemma \ref{lem:Minimizer} or Proposition \ref{prop:Minimizer} respectively yields $\phi(\xi^*)=\phi^*$ and $\nabla\phi(\xi^*)=0$, where $\lVert\sqrt SA\sqrt S/N\rVert_2<4\frac{q-1}q$ holds for $N$ sufficiently large. After plugging this into \eqref{eq:MEF} and undoing the change of variables, we arrive at
\begin{align}\label{eq:LimitMGF}
&\lim_{N\to\infty} \int \exp\left[ t^T\big(Z+ \sqrt{\mathcal S}(m-v)\big)\right]d\IP\nonumber\\
 &=\lim_{N\to\infty}c_N  e^{- N\phi^*}\left(\int_{\sqrt{N\mathcal S ^{-1}}B_{r\sqrt N}(0)} \exp\left[-\tfrac 1 2 x^T\sqrt{\frac{\mathcal S}N} H_\phi (\tilde\xi)\sqrt{\frac{\mathcal S}N}x+t^Tx\right]d^{sq}x+\mathcal O (e^{-\epsilon N})\right)\nonumber\\
 &=\lim_{N\to\infty}\left(c_N  e^{- N\phi^*}\left(\sqrt{(2\pi)^{sq}\det(\Theta)}+\mathcal O (e^{-\epsilon N})\right)\right)\int \exp\left[t^Tx\right]d\mathcal N(0,\Theta)(x),
 \end{align}
where $\Theta = (\sqrt{\Gamma} H_\phi (\xi^*)\sqrt{ \Gamma})^{-1}$. The last step follows from the dominated convergence theorem with an integrable majorant that exists due to \eqref{eq:posdef} and pointwise convergence holds, since the intermediate point converges $\tilde\xi\to\xi^*\to\Gamma\tfrac 1 q \mathbf 1_{sq}$ as $N\to\infty$ for each fixed $x$. In particular we obtain the normalizing constant from setting $t=0$ 
\begin{align*}
 \lim_{N\to\infty}c_N  e^{-N\phi^*}=\sqrt{\det(\Theta^{-1})/(2\pi)^{sq}}.
 \end{align*}
 
The moment-generating function of $Z$ is given by $t\mapsto \exp[\tfrac 12 t^TN(\sqrt{ \mathcal S}\mathcal A\sqrt{ \mathcal S})^{-1}t]$, which converges to $\exp[\tfrac 12 t^T(\sqrt{ \Gamma}\mathcal A\sqrt{\Gamma})^{-1}t]$ as $N\to\infty$. Since $Z$ is independent from $m$, the moment generating functions factorize and we conclude our claim
\begin{align*}
 \lim_{N\to\infty}\IE\exp[ t^T \sqrt{\mathcal S}(m-N\mathcal S ^{-1}\xi^*)]=\exp\left[\tfrac 1 2 t^T \left( \Gamma^{-1/2} (\lim_NH_\phi (\xi^*)^{-1}-\mathcal A^{-1})\Gamma^{-1/2}\right)t\right].
\end{align*}
The explicit representation of the Hessian has been evaluated in \eqref{eq:Hessianexplicit} already, hence
\begin{align*}
    \lim_{N\to\infty} H_\phi(\xi^*) = A\otimes I_{q}-\frac{A\operatorname{diag}(\gamma) A}{q}\otimes(I_q-\tfrac 1 q\mathbf 1_{q\times q}).
\end{align*}
In the case of asymptotically equal block sizes $\gamma=1/s$, we have $\xi^*=\frac 1 {sq}\mathbf 1 _{sq}$, $\Gamma=\frac 1 s I_{sq}$.
\end{proof}

Let us now derive the rotated CLT, Theorem \ref{thm:CLTrot}, from the CLT's with degenerate multivariate Gaussian limit distribution.

\begin{proof}[Proof of Theorem \ref{thm:CLTrot}]
The matrix $\tilde R\in SO(q)$ is given by
\begin{align*}
    \tilde R=\left(\begin{matrix}
    1-\frac{1}{q+\sqrt q} &-\frac{1}{q+\sqrt q} & \cdots & -\frac{1}{q+\sqrt q} & -\frac{1}{\sqrt q} \\
        -\frac{1}{q+\sqrt q} &1-\frac{1}{q+\sqrt q} & \ddots & \vdots& \vdots \\
        \vdots & \ddots& \ddots &  -\frac{1}{q+\sqrt q} &\vdots \\
        -\frac{1}{q+\sqrt q} & \cdots & -\frac{1}{q+\sqrt q} &1-\frac{1}{q+\sqrt q} & -\frac{1}{\sqrt q} \\
        \frac 1{\sqrt q}&\cdots&\cdots&\cdots&\frac 1 {\sqrt q}
\end{matrix}\right).
\end{align*}
By the definition of $\mathcal S$ it follows from Theorem \ref{thm:CLT} or Theorem \ref{thm:CLTgeneral} that 
\begin{align*}
\hat m=\mathcal  R\sqrt{\mathcal S}(m-\tfrac 1 q\mathbf 1 _{sq})\Rightarrow \tilde W
\end{align*}
where we applied the continuous mapping theorem. The limit $\tilde W$ has Gaussian distribution with covariance matrix $\mathcal  R\Sigma\mathcal  R^T$. Using the Kronecker product structure $\mathcal A=A\otimes I_q$ and $\Gamma=\operatorname{diag}(\gamma)\otimes I_q$, we rewrite
\begin{align*}
\Sigma&=\big(\sqrt\Gamma\mathcal A\sqrt\Gamma\big)^{-1}\left(\Big(I_{sq}+\sqrt\Gamma\mathcal A\sqrt\Gamma\big( I_s\otimes(\tfrac 1{q^2}\mathbf 1_{q\times q}-\tfrac 1 q I_q)\big)\Big)^{-1}-I_{sq}\right)\\
&=\big(\sqrt\Gamma\mathcal A\sqrt\Gamma\big)^{-1}\Big(I_{sq}+\sqrt\Gamma\mathcal A\sqrt\Gamma\big( I_s\otimes(\tfrac 1{q^2}\mathbf 1_{q\times q}-\tfrac 1 q I_q)\big)\Big)^{-1}\sqrt\Gamma\mathcal A\sqrt\Gamma\big( I_s\otimes(\tfrac 1 q I_q-\tfrac 1{q^2}\mathbf 1_{q\times q})\big)
\\
&=\left(qI_{sq}-\sqrt \Gamma\mathcal A  \sqrt\Gamma\left(I_s\otimes\big(I_q-\tfrac 1 q\mathbf 1_{q\times q}\big)\right)\right)^{-1}\cdot\left(I_s\otimes\left(I_q-\tfrac 1 q\mathbf 1_{q\times q}\right)\right)
\end{align*}
From $I_s\otimes\left(I_q-\tfrac 1 q\mathbf 1_{q\times q}\right)\cdot \mathcal R^T= \mathcal R^T$ it follows 
\begin{align*}
&\mathcal  R\Sigma\mathcal  R^T\cdot\left(qI_{s(q-1)}-\operatorname{diag(\gamma)}^{1/2}A \operatorname{diag(\gamma)}^{1/2}\otimes I_{q-1}\right)\\
&=\mathcal R \left(qI_{sq}-\sqrt\Gamma\mathcal A \sqrt\Gamma\left(I_s\otimes\big(I_q-\tfrac 1 q\mathbf 1_{q\times q}\big)\right)\right)^{-1}\mathcal R^T\cdot\left(qI_{s(q-1)}-\operatorname{diag(\gamma)}^{1/2}A \operatorname{diag(\gamma)}^{1/2}\otimes I_{q-1}\right)\\
&=\mathcal R\left(qI_{sq}-\sqrt\Gamma\mathcal A \sqrt\Gamma\left(I_s\otimes\big(I_q-\tfrac 1 q\mathbf 1_{q\times q}\big)\right)\right)^{-1}\cdot\left(q\mathcal R ^T-\sqrt\Gamma\mathcal A \sqrt\Gamma\left(I_s\otimes\big(I_q-\tfrac 1 q\mathbf 1_{q\times q}\big)\right)\mathcal R^T\right)\\
&=\mathcal R\mathcal R ^T=I_{s}\otimes I_{q-1}=I_{s(q-1)}.
\end{align*}
On the other hand it holds $\mathcal R\mathcal R ^T=I_{q-1}\otimes I_s=I_{(q-1)s}$, and thus we have actually shown that the covariance matrix is given by
\begin{align}
\mathcal  R\Sigma\mathcal  R^T&=\left(qI_{s(q-1)}-\operatorname{diag(\gamma)}^{1/2}A \operatorname{diag(\gamma)}^{1/2}\otimes I_{q-1}\right)^{-1}\nonumber\\
&=\left(q-\operatorname{diag(\gamma)}^{1/2}A \operatorname{diag(\gamma)}^{1/2}\right)^{-1}\otimes I_{q-1}\label{eq:rotatedcovariance}
\end{align}
and is non-singular. In the case of asymptotically equal block sizes, this equals $(q-A/s)^{-1}\otimes I_{q-1}$.
\end{proof}

At last, we turn to the

\begin{proof}[Proof of Theorem \ref{thm:MDP}]
We follow the steps of the proofs of the CLT's. Let us first consider the usual magnetization and rotate it only in the last step, when it is necessary. We are interested in the asymptotics of the moment generating function at any $t\in\IR^{sq}$, i.e.
 \begin{align*}
  \int \exp\left[ N^{1-2\theta }t^T\big(Z+ \mathcal S^{\theta}(m-v)\big)\right]d\IP=c_N \int\exp\left[- N \phi\left(\xi ^*+\left(\frac{\mathcal S}{ N}\right)^{1-\theta}\frac x { N^{\theta}} \right)+N^{1-2\theta}t^Tx\right]d^{sq}x
 \end{align*}
 for $v=N\mathcal S ^{-1}\xi^*$ and $Z\sim \mathcal N \left(0,N( \mathcal S^{1-\theta}\mathcal A \mathcal S^{1-\theta})^{-1}\right)$. Repeating the tail estimate along the lines of \eqref{eq:Flanken} with the change of variables $\xi=\xi^*+(\mathcal S/N)^{1-\theta}x/N^{\theta}$, we obtain
 \begin{align}\label{eq:Flanken2}
&\int_{B^c_r(\xi^*)} \exp\left[- N\big(\phi(\xi)-\phi^*\big)+t^TN^{1-\theta}(N\mathcal S ^{-1})^{1-\theta}(\xi-\xi^*)\right]d^{sq}\xi\\
&\le \int_{B^c_r(\xi^*)\cap B_R(0)} \exp\left[-N\delta+cN^{1-\theta} R^2\right]d^{sq}\xi+e^{N\phi^\ast}\int_{B^c_r(\xi^*)\cap B^c_R(0)} \exp\left[-\frac N 3 \xi^T\mathcal A\xi+ c N^{1-\theta}\norm \xi ^2\right]d^{sq}\xi\nonumber\\
&\le Ce^{-\epsilon N}\nonumber
\end{align}
  for some $r,R,c,C,\delta,\epsilon>0$. Let us denote $J_N(x)=\left(\frac{\mathcal S}N\right)^{1-\theta}H_\phi (\tilde\xi)\left(\frac{\mathcal S}N\right)^{1-\theta}$ for shorter notation, where $\tilde \xi$ is the intermediate point between $\xi^*$ and $\xi$ in the Lagrange remainder of the Taylor expansion \eqref{eq:Taylor}.
 Analogously to the derivation \eqref{eq:LimitMGF}, we obtain for all $t\in\IR ^{sq}$
\begin{align}
 &\int \exp\left[ N^{1-2\theta }t^T\big(Z+ \mathcal S^{\theta}(m-v)\big)\right]d\IP\nonumber\\
 &\sim c_N  e^{- N\phi^*}\int_{U_N} \exp\left[-\frac {N^{1-2\theta}} 2 x^TJ_N(x)x+N^{1-2\theta}t^Tx\right]d^{sq}x\nonumber\\
 &=c_Ne^{-N\phi^*} \int_{U_N}\exp\left[-\frac {N^{1-2\theta}} 2(x-J_N(x)^{-1} t)^TJ_N(x)(x-J_N(x)^{-1} t)\right]  \exp\left[\frac{N^{1-2\theta}}2t^TJ_N(x)^{-1}t\right]d^{sq}x\label{eq:LimitMGF2}
\end{align}
with $U_N=(\mathcal S/N)^{\theta-1}B_{r N^{\theta}}(0)$. Suppose for a moment that $J_N(x)$ can be replaced by its limit $J_\infty=\lim_N J_N(x)=\Gamma^{1-\theta} H_\phi (\xi^*) \Gamma^{1-\theta}$. Then it would follow from dominated convergence with an integrable majorant from \eqref{eq:posdef} and $ U_N\to\IR^{sq}$ that
\begin{align}\label{eq:LimitMGFresult}
\lim_{N\to\infty}\frac{1}{N^{1-2\theta}}\log \int \exp\left[ N^{1-2\theta }t^T\big(Z+ \mathcal S^{\theta}(m-v)\big)\right]d\IP=\frac 1 2 t^T (\Gamma^{1-\theta} H_\phi (\xi^*) \Gamma^{1-\theta})^{-1}t.
\end{align}
This is immediate the previous case of $\theta =1/2$, because the integrand in \eqref{eq:LimitMGF2} was convergent. In order to justify this limit, we will split the integration area $U_N$ of \eqref{eq:LimitMGF2} into a compact set and the rest. First choose $t=0$, perform a change of variables $y=N^{1/2-\theta}x$ and use dominated convergence with an integrable majorant from \eqref{eq:posdef}, which implies the asymptotic 
\begin{align}\label{eq:normalizingconstant}
c_Ne^{-N\phi^*}N^{(\theta-1/2)sq}= c+o(1).
\end{align}
Hence the normalizing constants are negligible after taking the logarithm and dividing by $N^{1-2\theta}$. It follows from \eqref{eq:posdef} and continuity of $H_\phi$ that $\norm{ H_\phi(\tilde \xi)^{-1}-H_\phi(\xi^*)^{-1}}=o (1)$ uniformly for $\xi\in B_{cN^{-\theta}}(\xi ^*)$, hence we have $J_N(x)^{-1}=J_\infty^{-1}+o(1)$ uniformly for $x$ in a compact set, say $x\in \bar B_\kappa(0)$. Thus,
\begin{align}\label{eq:CompactDomain}
 &\int_{\bar B_\kappa(0)}\exp\left[-\frac {N^{1-2\theta}} 2(x-J_N(x)^{-1} t)^TJ_N(x)(x-J_N(x)^{-1} t)\right] \exp\left[\frac{N^{1-2\theta}}2t^TJ_N(x)^{-1}t\right]d^{sq}x\nonumber\\
 &=\exp\left[\frac{N^{1-2\theta}}2t^T(J_\infty^{-1}+o(1))t\right]\int_{\bar B_\kappa(0)}\exp\left[-\frac {N^{1-2\theta}} 2(x-(J_\infty^{-1}+o(1)) t)^TJ_N(x)(x-(J_\infty^{-1}+o(1)) t)\right] d^{sq}x
\end{align}
with the integral being of order $I_N\asymp (N^{(\theta-1/2)sq})$, similar to \eqref{eq:normalizingconstant}. It follows from \eqref{eq:posdef} and continuity of $H_\phi$ that $\norm{ H_\phi(\tilde \xi)^{-1}-H_\phi(\xi^*)^{-1}}=\mathcal O (1)$ uniformly for $\xi\in B_r(0)$, hence after all we have $J_N(x)^{-1}=J_\infty^{-1}+\mathcal O (1)$ uniformly for $x\in U_N$. Thus we can bound
\begin{align*}
  &\int_{U_N\setminus \bar B_\kappa(0)}\exp\left[-\frac {N^{1-2\theta}} 2(x-J_N(x)^{-1} t)^TJ_N(x)(x-J_N(x)^{-1} t)\right] \exp\left[\frac{N^{1-2\theta}}2t^TJ_N(x)^{-1}t\right]d^{sq}x\\
  &\lesssim\int_{U_N\setminus \bar B_\kappa(0)}\exp\left[-c N^{1-2\theta}\norm{x-J_\infty^{-1}t+t\mathcal O (1)}^2\right] \exp\left[\frac{\varrho}2 N^{1-2\theta} \norm{t}^2\right]d^{sq}x,
\end{align*}
where $c>0$ comes from the positive definiteness shown in Lemma \ref{lem:posdef} and $\varrho=\mathcal O (1)$ is bounded by the spectral radius of $J_N(x)=J_\infty+\mathcal O (1)$. Choosing $\kappa>0$ sufficiently large, it follows
\begin{align*}
 &\int_{U_N\setminus \bar B_\kappa(0)}\exp\left[-c N^{1-2\theta}\norm{x-J_\infty^{-1}t+t\mathcal O (1)}^2\right] \exp\left[\frac{\varrho}2 N^{1-2\theta} \norm{t}^2\right]d^{sq}x\\
 &\lesssim \int_{U_N\setminus \bar B_\kappa(0)}\exp\left[-\frac c 2 N^{1-2\theta}\norm{x}^2\right] d^{sq}x\lesssim  N^{sq\theta}\exp\left[-\frac c 2 N^{1-2\theta}\kappa^2\right]=o(1).
\end{align*}
Therefore, combining this part,  \eqref{eq:CompactDomain} and \eqref{eq:normalizingconstant}, taking the logarithm and dividing by $N^{1-2\theta}$, we have shown
\begin{align*}
 &\frac {1}{N^{1-2\theta}}\log\left( c_N  e^{- N\phi^*}\int_{U_N} \exp\left[-\frac {N^{1-2\theta}} 2 x^TJ_N(x)x+N^{1-2\theta}t^Tx\right]d^{sq}x\right)\\
 &=\frac {1}{N^{1-2\theta}}\log\left[cN^{(1/2-\theta)sq}\left(\exp\left[\frac{N^{1-2\theta}}2t^T(J_\infty^{-1}+o(1))t\right]\cdot I_N+o(1) \right)\right]\\
 &= \frac 1 2 t^T (\Gamma^{1-\theta} H_\phi (\xi^*) \Gamma^{1-\theta})^{-1}t+o(1).
\end{align*}
Consequently, we have shown that \eqref{eq:LimitMGFresult} holds.

The moment generating function of $\sqrt{N^{1-2\theta}}Z$ equals $\exp[\tfrac 12t^TN( (\mathcal S/N)^{1-\theta}\mathcal A (\mathcal S/N)^{1-\theta})^{-1}t]$, hence
\begin{align*}
    \frac 1 {N^{1-2\theta}}\log\IE\left(\exp[N^{1-2\theta}t^T Z]\right)=\tfrac 12t^TN( (\mathcal S/N)^{1-\theta}\mathcal A (\mathcal S/N)^{1-\theta})^{-1}t\to \tfrac 1 2 t^T (\Gamma^{1-\theta}\mathcal A\Gamma^{1-\theta})^{-1}t
\end{align*}
as $N\to \infty$. Thus we obtain 
\begin{align*}
\lim_{N\to\infty}\frac{1}{N^{1-2\theta}}\log \int \exp\left[ N^{1-2\theta }t^T \mathcal S^{\theta}(m-v)\right]d\IP=\frac 1 2 t^T \left(\Gamma^{\theta-1} (H_\phi (\xi^*)^{-1}-\mathcal A ^{-1}) \Gamma^{\theta-1}\right)t.
\end{align*}
Now, consider $t=\mathcal R ^T \tilde t$ and we arrive at
\begin{align*}
\lim_{N\to\infty}\frac{1}{N^{1-2\theta}}\log \int \exp\left[ N^{1-2\theta }\tilde t^T\mathcal R\mathcal S^{\theta}(m-v)\right]d\IP=\frac 1 2 \tilde t^T \mathcal R\left(\Gamma^{\theta-1} (H_\phi (\xi^*)^{-1}-\mathcal A ^{-1}) \Gamma^{\theta-1}\right)\mathcal R ^T \tilde t=:\Lambda^*(\tilde t).
\end{align*}
Taking the same route as in the proof of Theorem \ref{thm:CLTrot}, we obtain
\begin{align*}
    \mathcal R\left(\Gamma^{\theta-1} (H_\phi (\xi^*)^{-1}-\mathcal A ^{-1}) \Gamma^{\theta-1}\right)\mathcal R ^T=\left(q\operatorname{diag(\gamma)}^{1-2\theta}-\operatorname{diag(\gamma)}^{1-\theta}A \operatorname{diag(\gamma)}^{1-\theta}\right)^{-1}\otimes I_{q-1}
\end{align*}
as in \eqref{eq:rotatedcovariance}, where we omit the analogous details. In particular this matrix is positive definite and hence the quadratic form $\Lambda^*$ is an essentially smooth function. The Gärtner-Ellis Theorem \cite[Theorem 2.3.6]{DemboZeituoni} now provides the large deviation principle of $\mathcal S^{\theta}(m-N\mathcal S^{-1}\xi^*)$ with speed $N^{1-2\theta}$ and good rate function given by the Legendre transform $(\Lambda^*)^*=\Lambda$. This is given by the quadratic form of half the inverse matrix and we wrote $t$ instead $\tilde t$ in the statement of the theorem.
\end{proof}


\end{document}